\documentclass{article}
\usepackage{style141}

\title{On Kummer characters arising from Galois actions on the pro-$p$ fundamental groups of once-punctured CM elliptic curves}
\author{Shun Ishii \\ \href{ishii.shun@keio.jp}{ishii.shun@keio.jp} }
\affil{{\small Department of Mathematics, Keio University, 3-14-1 Hiyoshi, Kouhoku-ku, Yokohama 223-8522, Japan.}}
\date{}

\begin{document}

\maketitle

\begin{abstract}
In this paper, we study certain characters arising from the Galois actions on the pro-$p$ geometric fundamental groups of once-punctured CM elliptic curves. These characters may be regarded as analogues of the Soul\'e characters, which arise from the Galois action associated to the thrice-punctured projective line. We study various arithmetic assumptions under which such characters are nontrivial or even surjective.

\end{abstract}

\tableofcontents

\section{Introduction}\label{1}

In this paper, we study certain characters arising from Galois actions on the pro-$p$ geometric fundamental groups (i.e. the maximal pro-$p$ quotients of the geometric \'etale fundamental groups) of once-punctured CM elliptic curves. 

These characters may be regarded as elliptic analogues of \emph{the Soul\'e characters}, which arise from the Galois action on the pro-$p$ geometric fundamental group $\pi_{1}^{\et}(\mathbb{P}^{1}_{\overline{\mathbb{Q}}} \setminus \{ 0,1,\infty \})^{(p)}$ of the thrice-punctured projective line. 

We briefly recall backgrounds. Let $p$ be an odd prime and $(\zeta_{n})_{n \geq 1}$ a sequence primitive $p^{n}$-th roots of unity in $\overline{\mathbb{Q}}$ satisfying $\zeta_{n+1}^{p}=\zeta_{n}$ for every $n$. For every odd integer $m \geq 3$, \emph{the $m$-th Soul\'e
character} is a $\mathrm{Gal}(\mathbb{Q}(\mu_{p^{\infty}})/\mathbb{Q})$-equivariant character
\[
\kappa_{m} \colon G_{\mathbb{Q}(\mu_{p^{\infty}})}^{\mathrm{ab}} \to \mathbb{Z}_{p}(m)
\] whose reduction modulo $p^{n}$, via Kummer theory, corresponds to $p^{n}$-th roots of  
\[
\prod_{1 \leq a \leq p^{n}, (a,p)=1}(1-\zeta_{n}^{a})^{a^{m-1}}
\]  for every $n \geq 1$. Note that this character depends on the choice of $(\zeta_{n})_{n \geq 1}$ only up to multiplication by an element of $\mathbb{Z}_{p}^{\times}$.

\begin{theorem}[Properties of the Soul\'e characters]\label{thm:fund}\hfill
\begin{enumerate}
\item $\kappa_{m}$ is nontrivial for every odd $m \geq 3$. 
\item $\kappa_{m}$ is surjective for every $m \geq 3$ such that $m \equiv 1 \bmod p-1$.
\item $\kappa_{m}$ is surjective for every odd $m \geq 3$ if and only if Vandiver's conjecture holds for $p$, i.e., the class number of the number field $\mathbb{Q}(\zeta_{p}+\zeta_{p}^{-1})$ is not divisible by $p$.
\end{enumerate}
\end{theorem}

For proofs, see Ichimura-Sakaguchi \cite{IS}. There the nontriviality of the Soul\'e characters are mainly resorted to two ingredients: the finiteness of the cohomology group $H^{2}_{\et}(\mathrm{Spec}(\mathbb{Z}[\frac{1}{p}]), \mathbb{Z}_{p}(m))$ for $m \geq 2$ proved by Soul\'e \cite[page 287, Corollaire]{So3} and Iwasawa main conjecture proved by Mazur-Wiles \cite{MZ}. As for the surjectivity, a key ingredient is a relation between the class number of $\mathbb{Q}(\zeta_{p}+\zeta_{p}^{-1})$ and the index of the group of cyclotomic units inside the group of units.

Let us briefly mention a relationship between the Soul\'e characters and the Galois action on the pro-$p$ geometric fundamental group $\pi_{1}^{\et}(\mathbb{P}^{1}_{\overline{\mathbb{Q}}} \setminus \{ 0,1,\infty \})^{(p)}$. In his pioneering work \cite{Ih1}, Ihara constructed and studied a homomorphism 
\[
\alpha_{0,3} \colon G_{\mathbb{Q}(\mu_{p^{\infty}})}^{\mathrm{ab}} \to \mathbb{Z}_{p}[[T_{1}, T_{2}]] 
\] called \emph{the universal power series for Jacobi sums} in a group-theoretical way from the Galois action on the maximal metabelian quotient of $\pi_{1}^{\et}(\mathbb{P}^{1}_{\overline{\mathbb{Q}}} \setminus \{ 0,1,\infty \})^{(p)}$. The Soul\'e characters are related to this power series in the following way:

\begin{theorem}[Ihara-Kaneko-Yukinari {\cite{IKY}}]
\[
\alpha_{0,3}(\sigma) = \sum_{m \geq 3 \colon \text{odd}} \frac{\kappa_{m}(\sigma)}{1-p^{m-1}} \sum_{i+j=m} \frac{U_{1}^{m_{1}}U_{2}^{m_{2}}}{m_{1}!m_{2}!}.
\] holds for every $\sigma \in G_{\mathbb{Q}(\mu_{p^{\infty}})}$. Here, we regard $\mathbb{Z}_{p}[[T_{1}, T_{2}]]$ as a subring of $\mathbb{Q}_{p}[[U_{1}, U_{2}]]$ via $U_{i} \coloneqq \log(1+T_{i})$ for $i=1, 2$. 
\end{theorem}

For more discussions of the Galois action on the pro-$p$ geometric fundamental group of $\mathbb{P}^{1}_{\mathbb{Q}} \setminus \{ 0,1,\infty \}$, see \cite{Ih6}, for example. 

\medskip

We move on to the case of once-punctured elliptic curves. In \cite{Na2}, Nakamura studied Galois actions on the pro-$p$ geometric fundamental group $\pi_{1}^{\et}(E_{\overline{K}} \setminus O)^{(p)}$ of a once-punctured elliptic curve $E \setminus O$ over a number field $K$ and defined an analogue of Ihara's power series
\[
\alpha_{1,1} \colon G_{K(E[p^{\infty}])}^{\mathrm{ab}} \to \mathbb{Z}_{p}[[T_{p}(E)]](1)
\] which is denoted by $\alpha$ in the original paper. Moreover, he found explicit Kummer-theoretic characterizations of coefficients of his power series. 

To state his result, we fix a basis $\omega_{1}=(\omega_{1,n})_{n\geq 1}$ and $\omega_{2}=(\omega_{2,n})_{n \geq 1}$ of $T_{p}(E)=\varprojlim_{n} E[p^{n}](\overline{\mathbb{Q}})$. We identify $\mathbb{Z}_{p}[[T_{p}(E)]]$ with $\mathbb{Z}_{p}[[T_{1}, T_{2}]]$ by $T_{i} =\omega_{i}-1$. We define the elliptic Soul\'e characters as follows:

\begin{definition}[cf. {\cite[(3.11.3)]{Na2}}]\label{ellSoule}
  For $\boldsymbol{m}=(m_{1}, m_{2}) \in \mathbb{Z}^{2}_{\geq 1}$ such that $m_{1}+m_{2}$ is even, 
  we define \emph{the $\boldsymbol{m}$-th elliptic Soul\'e character}
  \[
  \kappa_{\boldsymbol{m}} \colon G_{K(E[p^{\infty}])}^{\rm{ab}} \to \mathbb{Z}_{p}
  \] to be the character whose reduction modulo $p^{n}$ corresponds to 
  \[
    \epsilon_{(m_{1}, m_{2}), n} \coloneqq \prod_{\substack{0 \leq a,b <p^{n} \\ p \nmid \gcd(a,b)}}
\left(
\prod_{k=0}^{n} \theta_{p}(a\omega_{1, n+1+k}+b\omega_{2, n+1+k})^{-p^{2k}}
\right)^{a^{m_{1}-1}b^{m_{2}-1}}
\] where $\theta_{p}$ is a rational function on the curve $E/K$ defined in Section \ref{2.1}, via Kummer theory. Here, note that we use the convention $0^{0} \coloneqq 1$.   
\end{definition}

\begin{theorem}[Nakamura  {\cite[Theorem (A)]{Na2}}]\label{thm:Nakamura}
\[
\alpha_{1,1}(\sigma)=\sum_{m \geq 2 \colon \text{even}}^{\infty}\frac{1}{1-p^{m}}\sum_{\substack{(m_{1}, m_{2}) \in \mathbb{Z}^{2}_{\geq 0} \\ m_{1}+m_{2}=m}}\kappa_{(m_{1}+1, m_{2}+1)}(\sigma)\frac{U_{1}^{m_{1}}U_{2}^{m_{2}}}{m_{1}! m_{2}!}
\] holds for every $\sigma \in G_{K(E[p^{\infty}])}$. Here, we regard $\mathbb{Z}_{p}[[T_{1}, T_{2}]]$ as a subring of $\mathbb{Q}_{p}[[U_{1}, U_{2}]]$ via $U_{i}=\log(1+T_{i})$ for $i=1,2$.
\end{theorem}

We have important remarks regarding the elliptic Soul\'e characters:

\begin{remark}
  (1) As in the case of the Soul\'e characters, the elliptic Soul\'e characters does depend on the choices of bases of $T_{\p}(E)$ and $T_{\bar{\p}}(E)$, only up to multiplication by $\mathbb{Z}_{p}^{\times}$ by definition. Since we work with the fixed basis $\{ \omega_{1}, \omega_{2} \}$, we do not mention this dependence hereafter.

  (2) Theorem \ref{thm:Nakamura} was originally stated via a certain basis of $T_{p}(E)$ which comes from a free basis of $\pi_{1}^{\mathrm{top}}(E(\mathbb{C}) \setminus O, \vec{0})$ where $\vec{0}$ is a tangential basepoint at the origin. By using such a topological basis, one can simplify Definition \ref{ellSoule} by using the fundamental theta function $\theta(z, \mathcal{L})$ associated to a lattice $\mathcal{L}$ for $E(\mathbb{C})$ \cite[\textsection 2]{Na2}. However, the proof given in \cite{Na2} works for every basis. 
\end{remark}

While properties of the Soul\'e characters are well understood, the fundamental properties of the elliptic Soul\'e characters, such as the surjectivity or even the nontriviality, are largely unknown. As a partial result, Nakamura \cite[(3.12)]{Na2} proved that certain linear combinations of them are nontrivial. 

In this paper, we consider the elliptic Soul\'e characters arising from CM elliptic curves over imaginary quadratic fields, and study the nontriviality and the surjectivity of them. 

\medskip

Let $K$ be an imaginary quadratic field of class number one and $(E, O)$ an elliptic curve over $K$ which has complex multiplication by the ring of integers $O_{K}$ of $K$\footnote{Note that the isomorphism class of $E$ over $\overline{\mathbb{Q}}$ does not depend on the choice of an elliptic curve with CM by $O_{K}$. Hence, geometrically, we consider a single elliptic curve for each $K$.}.

Let $p \geq 5$ be a prime such that $E$ has potentially good ordinary reduction at primes above $p$. The ideal $(p) \subset O_{K}$ splits into a product of two primes $\p=(\pi)$ and its complex conjugate $\bar{\p}=(\bar{\pi})$ where $p=\pi \bar{\pi}$. The Tate module $T_{p}(E)$ also splits into the direct sum of \[
  T_{\p}(E) \coloneqq \varprojlim_{n} E[\p^{n}](\overline{\mathbb{Q}}) \quad \text{and} \quad T_{\p}(E) \coloneqq \varprojlim_{n} E[\p^{n}](\overline{\mathbb{Q}})
  \] and we have the corresponding characters 
\[
\chi_{1} \colon G_{K} \to \mathrm{Aut}(T_{\p}(E)) \cong \mathbb{Z}_{p}^{\times} \quad \text{and} \quad
\chi_{2} \colon G_{K} \to \mathrm{Aut}(T_{\bar{\p}}(E)) \cong \mathbb{Z}_{p}^{\times}.
\] 

We fix a basis $\omega_{1}=(\omega_{1,n})_{n \geq 1}$ and $\omega_{2}=(\omega_{2,n})_{n \geq 1}$ of $T_{p}(E)$ so that $\omega_{1}$ (resp. $\omega_{2}$) generates $T_{\p}(E)$ (resp. $T_{\bar{\p}}(E)$). 

\medskip

Since the homomorphism $\alpha_{1,1}$ is equivariant and the space of degree $m$ homogeneous polynomials in $\mathbb{Q}[[T_{1}, U_{2}]]$ is naturally isomorphic $\mathrm{Sym}^{m}\left( T_{p}(E) \otimes \mathbb{Q}_{p} \right)$, the character $\kappa_{\boldsymbol{m}}$ gives a $\mathrm{Gal}(K(E[p^{\infty}])/K)$-equivariant homomorphism
\[
\kappa_{\boldsymbol{m}} \colon G_{K(E[p^{\infty}])}^{\mathrm{ab}} \to \mathbb{Z}_{p}(\boldsymbol{m}) \coloneqq \mathbb{Z}_{p}(\chi_{1}^{m_{1}}\chi_{2}^{m_{2}}).
\]

The next theorem gives an analogue of Theorem \ref{thm:fund}: in the following, we denote the ray class field of $K$ modulo $p$ (resp. $\p$) by $K(p)$ (resp. $K(\p)$).

\begin{theorem}[Properties of the elliptic Soul\'e characters, Theorem \ref{non_vanishing} and Sections \ref{5.1}-\ref{5.4}]\label{thm:main}

Let $I \coloneqq \{ \boldsymbol{m}=(m_{1}, m_{2}) \in \mathbb{Z}_{ \geq 1}^{2} \setminus \{ (1,1) \} \mid m_{1} \equiv m_{2} \bmod \lvert O_{K}^{\times} \rvert \}$ and $\boldsymbol{m} \in I$.
\begin{enumerate}

\item $\kappa_{\boldsymbol{m}}$ is nontrivial if $H^{2}_{\et}(\mathrm{Spec}(O_{K}[\frac{1}{p}]), \mathbb{Z}_{p}(\boldsymbol{m}))$ is finite.

\item $\kappa_{\boldsymbol{m}}$ is not surjective if $\boldsymbol{m} \geq (2,2)$ and $\boldsymbol{m} \equiv (1,1) \mod p-1$. 

\item $\kappa_{\boldsymbol{m}}$ is surjective for every $\boldsymbol{m} \in I$ such that $\boldsymbol{m} \geq (2,2)$ and  $\boldsymbol{m} \not \equiv (1,1) \bmod p-1$ if and only if the class number of $K(p)$ is not divisible by $p$ and there exists a unique prime of $K(p)$ above $\p$.

\item $\kappa_{(m,1)}$ (resp. $\kappa_{(1,m)}$) is surjective if $m \in \mathbb{Z}_{>1}$ satisfies  $m \equiv 1 \bmod p-1$. 

\item $\kappa_{(m,1)}$ (resp. $\kappa_{(1,m)}$) is surjective for every $(m,1) \in I$ (resp. $(1,m) \in I$) if and only if the class number of $K(\p)$ (resp. $K(\bar{\p})$) is not divisible by $p$.
\end{enumerate}
\end{theorem}

\begin{remark}
(1) The finiteness of $H^{2}_{\et}$ in Theorem \ref{thm:main} (1) is a special case of a conjecture of Jannsen \cite[Conjecture 1]{Ja}. In Section \ref{4.2}, we discuss some assumptions under which this finiteness holds.

(2) Kucuksakalli \cite[Theorem 2]{Ku} proved that the class number of $K(\p)$ is divisible by $p$ when $K=\mathbb{Q}(\sqrt{-163})$ and $\p$ is a prime above $p=307$. In particular, some elliptic Soul\'e characters of the form $\kappa_{(m,1)}$ are not surjective.
\end{remark}

This paper is organized as follows. In Section \ref{2}, we discuss certain properties of elliptic units which are frequently used in the following sections. In Section \ref{3}, we modify elliptic Soul\'e characters to obtain certain cocycles coming from norm compatible systems of elliptic units in Section \ref{4}. Such a construction originates from Soul\'e \cite{So1} \cite{So2} and is also studied by Kings \cite[2.2.1]{Ki3} in a similar situation. In Section \ref{4.1}, we prove Theorem \ref{thm:main} (1) by adopting an argument of Kings \cite[Proposition 5.2.5]{Ki3}. In Section \ref{4.2}, we discuss various conditions which imply the finiteness of $H^{2}_{\et}$. In Section \ref{5}, using a relation between the class numbers of ray class fields of $K$ and elliptic units, we prove (2)-(5) of Theorem \ref{thm:main}.

\section*{Notations}
\hfill

In the rest of this paper, $K$ is an imaginary quadratic field of class number one and $(E, O)$ is an elliptic curve over $K$ which has complex multiplication by $O_{K}$ as in Section \ref{1}.

Moreover, $p \geq 5$ is a prime such that $E$ has potentially good ordinary reduction at primes. The ideal $(p)$ splits into a product of $\p=(\pi)$ and its complex conjugate $\bar{\p}=(\bar{\pi})$ where $p=\pi \bar{\pi}$. We fix a basis $\omega_{1}=(\omega_{1,n})_{n \geq 1}$ of $T_{\p}(E)$ and $\omega_{2}=(\omega_{2,n})_{n \geq 1}$ of $T_{\bar{\p}}(E)$. We have the corresponding characters
\[
\chi_{1} \colon G_{K} \to \mathrm{Aut}(T_{\p}(E)) \cong \mathbb{Z}_{p}^{\times} \quad \text{and} \quad
\chi_{2} \colon G_{K} \to \mathrm{Aut}(T_{\bar{\p}}(E)) \cong \mathbb{Z}_{p}^{\times}.
\] 

\medskip

{\bf Indexes.} For $\boldsymbol{m}=(m_{1}, m_{2}) \in \mathbb{Z}^{2}$ and an integer $n$, we write $\boldsymbol{m} \equiv 0 \bmod n$ if $\boldsymbol{m} \in n\mathbb{Z}^{2}$. For any two pairs $\boldsymbol{m}=(m_{1}, m_{2})$ and $\boldsymbol{n}=(n_{1}, n_{2})$, we write $\boldsymbol{m} \geq \boldsymbol{n}$ if $m_{i} \geq n_{i}$ holds for $i=1,2$. Similarly, we write $\boldsymbol{m}>\boldsymbol{n}$ if both $\boldsymbol{m} \geq \boldsymbol{n}$ and $\boldsymbol{m} \neq \boldsymbol{n}$ hold. In the following, we denote $(1,1)$ by~$\boldsymbol{1}$. 

\medskip

{\bf Profinite Groups.} For a profinite group $G$, we denote the maximal abelian quotient of $G$ by $G^{\rm{ab}}$. Moreover, we denote the maximal pro-$p$ quotient of $G$ by $G^{(p)}$. 

\medskip

{\bf Number Fields. } We fix an algebraic closure $\overline{\mathbb{Q}}$ of $\mathbb{Q}$ and every number field is considered to be a subfield of $\overline{\mathbb{Q}}$. For a number field $F$, we denote the absolute Galois group $\mathrm{Gal}(\overline{\mathbb{Q}}/F)$ of $F$ by $G_{F}$, the ring of integers of $F$ by $O_{F}$ and the group of roots of unity in $F$ by $\mu(F)$. For an integer $m \geq 1$, we denote the group of $m$-th roots of unity in $\overline{\mathbb{Q}}$ by $\mu_{m}$.

For a set $S$ of nonarchimedean places of $F$, we denote the ring of $S$-integers of $F$ by $O_{F,S}$. For a nonarchimedean place $v$ of $F$, we denote the $v$-adic completions of $K$ (resp. $O_{F}$) by $F_{v}$ (resp. $O_{F_{v}}$). 

For a subfield $F$ of $\overline{\mathbb{Q}}$, the set of primes of $F$ above $p$ (resp. $\p$, $\bar{\p}$, if $F$ contains $K$) is denoted by $S_{p}(F)$ (resp. $S_{\p}(F)$, $S_{\bar{\p}}(F)$). Moreover, we denote the maximal pro-$p$ extension of $F$ unramified outside $S_{p}(F)$ (resp. $S_{\p}(F)$, $S_{\bar{\p}}(F)$) by $F_{S_{p}}(p)$ (resp. $F_{S_{\p}}(p)$, $F_{S_{\bar{\p}}}(p)$). We drop $F$ from e.g. $S_{p}(F)$ if the field is clear from context.

\medskip

{\bf Imaginary Quadratic Fields.} We denote the order of $O_{K}^{\times}$ by $w_{K}$. For a nonzero ideal $\mathfrak{m}$ of $O_{K}$, we denote the order of $O_{K}/\mathfrak{m}$ by $N\mathfrak{m}$ and the ray class field of $K$ modulo $\mathfrak{m}$ by $K(\mathfrak{m})$. We write $K(\mathfrak{m}^{\infty}) \coloneqq \cup_{n \geq 1} K(\mathfrak{m}^{n})$. Moreover, for a nonzero ideal $\mathfrak{a}$ of $O_{K}$ relatively prime to $\mathfrak{m}$, we denote the Artin symbol $(\mathfrak{a}, K(\mathfrak{m})/K) \in \mathrm{Gal}(K(\mathfrak{m})/K)$ by $\sigma_{\mathfrak{a}}$.

Since we assume $p$ splits in $K$, the natural inclusions $ \mathbb{Z}_{p} \hookrightarrow O_{K_{\p}}$ and $\mathbb{Z}_{p} \hookrightarrow O_{K_{\bar{\p}}}$ are isomorphisms. Hence we have two ring homomorphisms 
\[i_{1} \colon O_{K} \hookrightarrow O_{K_{\p}} \xrightarrow{\sim} \mathbb{Z}_{p} \quad \text{and} \quad 
i_{2} \colon O_{K} \hookrightarrow O_{K_{\bar{\p}}} \xrightarrow{\sim} \mathbb{Z}_{p}.
\] For $\boldsymbol{m}=(m_{1}, m_{2}) \in \mathbb{Z}_{\geq 0}^{2}$, we denote $i_{1}^{m_{1}}i_{2}^{m_{2}} \colon O_{K} \to \mathbb{Z}_{p}$ by $i^{\boldsymbol{m}}$.

\medskip
 
{\bf CM Elliptic Curves.} For $\alpha \in O_{K}$, we shall write the corresponding endomorphism of $E$ by $[\alpha]$. For an ideal $\mathfrak{m}$ of $O_{K}$, the set of $\overline{\mathbb{Q}}$-points of the $\mathfrak{m}$-torsion subgroup scheme $E[\mathfrak{m}]$ determines an injective homomorphism
\[
\mathrm{Gal}(K(E[\mathfrak{m}])/K) \hookrightarrow \mathrm{Aut}(E[\mathfrak{m}](\overline{\mathbb{Q}})) \cong (O_{K}/\mathfrak{m})^{\times}.
\] and this homomorphism induces an isomorphism
\[
\mathrm{Gal}(K(\mathfrak{m})/K) \xrightarrow{\sim} (O_{K}/\mathfrak{m})^{\times}/\mathrm{im}(O_{K}^{\times})
\] which does not depend on the choice of $E$. 

\medskip

For $\boldsymbol{m}=(m_{1}, m_{2}) \in \mathbb{Z}^{2}$, we write
\[
  \chi^{\boldsymbol{m}} \coloneqq \chi_{1}^{m_{1}}\chi_{2}^{m_{2}} \colon G_{K} \to \mathbb{Z}_{p}^{\times}.
\] Note that, if $ m_{1} \equiv m_{2} \bmod w_{K}$, then $\chi^{\boldsymbol{m}}$ factors through $\mathrm{Gal}(K(E[p^{\infty}])/K) \to \mathrm{Gal}(K(p^{\infty})/K)$. 

For a $\mathbb{Z}_{p}$-module $M$ on which $G_{K}$ acts, we denote the $\chi^{\boldsymbol{m}}$-twist of $M$ by $M(\boldsymbol{m})$. We identify $T_{\p}(E)$  with $\mathbb{Z}_{p}(1,0)$ and $T_{\bar{\p}}$ with $\mathbb{Z}_{p}(0,1)$ via our chosen basis, respectively. Note that $\mathbb{Z}_{p}(m,m)$ is simply the $m$-th Tate twist since $\chi^{\boldsymbol{1}}=\chi_{1}\chi_{2}$ is the $p$-adic cyclotomic character.

\section{Preliminaries on elliptic units}\label{2}

In this section, we define a certain rational function $\theta_{\mathfrak{a}}$ on $E$ associated to a nontrivial ideal $\mathfrak{a}$ of $O_{K}$ and study arithmetic properties of its special values evaluated at torsion points of $E$. The function $\theta_{\mathfrak{a}}$ is also a quotient of the so-called fundamental theta function, cf. de Shalit \cite[Chapter II]{dS}. Our main references are de Shalit \cite{dS} and Rubin \cite{Ru1} \cite{Ru2}. 

\subsection{The function $\theta_{\mathfrak{a}}$}\label{2.1}

\begin{definition}[{\cite[Chapter $\rm{I\hspace{-.01em}I}$ \textsection 2, 2.3]{dS}}]
Let $\mathfrak{a}$ be a nontrivial ideal of $O_{K}$. We define a rational function $\theta_{\mathfrak{a}}$ on $E$ by
\[
\theta_{\mathfrak{a}} = \alpha^{-12}\Delta(E)^{N\mathfrak{a}-1}\prod_{\nu \in E[\mathfrak{a}]\setminus O}(x-x(\nu))^{-6}.
\]
Here, $\alpha$ is a generator of $\mathfrak{a}$, and $\Delta(E)$ is the discriminant of the chosen Weierstrass model of $E$. Moreover, $x \colon E \to \mathbb{P}^{1}$ is a finite morphism of degree two corresponding to the $x$-coordinate of $E$.
\end{definition}

The function $\mathfrak{a}$ is defined over $K$ with $\mathrm{div}(\theta_{\mathfrak{a}})=12(N\mathfrak{a}[O]-\sum_{P \in E[\mathfrak{a}]}[P])$. Moreover, $\theta_{\mathfrak{a}}$ is invariant under the action of $\mathrm{Aut}(E)$. The following two propositions are frequently used in this section:

\begin{proposition}[The distribution relation {\cite[Chapter $\rm{I\hspace{-.01em}I}$ \textsection 2, 2.3 Proposition]{dS}}]\label{distribution_relation}
Let $\mathfrak{a}$ and $\mathfrak{b}$ be nontrivial ideals of $O_{K}$ relatively prime to each other, and let $\beta$ be a generator of $\mathfrak{b}$. Then we have
\[
\prod_{\nu \in E[\mathfrak{b}]} \theta_{\mathfrak{a}}(\tau+\nu)=\theta_{\mathfrak{a}}(\beta \tau).
\] 
\end{proposition}

\begin{proposition}[{\cite[Chapter $\rm{I\hspace{-.01em}I}$ \textsection 2, 2.4 Proposition]{dS}}, see also {\cite[Theorem 7.4]{Ru1}}]\label{elliptic_unit}
Let $\mathfrak{m}$ and $\mathfrak{a}$ be nontrivial ideals of $O_{K}$ relatively prime to each other, and $\tau$ a primitive $\mathfrak{m}$-torsion point of $E$. Then, the following assertions hold:
\begin{enumerate}
\item $\theta_{\mathfrak{a}}(\tau) \in K(\mathfrak{m})$.
\item For every ideal $\mathfrak{c}=(c)$ of $O_{K}$ relatively prime to $\mathfrak{m}$, we have
\[
\theta_{\mathfrak{a}}(\tau)^{\sigma_{\mathfrak{c}}}=\theta_{\mathfrak{a}}(c\tau)=\theta_{\mathfrak{a}\mathfrak{c}}(\tau)\theta_{\mathfrak{c}}(\tau)^{-N\mathfrak{a}}.
\]
\item $\theta_{\mathfrak{a}}(\tau)$ is a unit unless $\mathfrak{m}$ is a power of a maximal ideal. If $\mathfrak{m}$ is a power of a maximal ideal $\mathfrak{l}$, then $\theta_{\mathfrak{a}}(\tau)$ is an $\mathfrak{l}$-unit.
\end{enumerate}
\end{proposition}

\begin{remark}\label{pthpower}
  Let $\tau \in \left( E[p^{\infty}]\setminus E[p] \right)(\overline{\mathbb{Q}})$. Then $\theta_{p}(\tau)$ is a $p$-unit in $K(p^{\infty})$. The proof of Proposition \ref{elliptic_unit} (3) given in \cite[Theorem 7.4]{Ru1} works in this case.  
\end{remark}

\begin{lemma}\label{p_to_a}
Let $\mathfrak{a}$ and $\mathfrak{b}$ be nontrivial ideals of $O_{K}$ relatively prime to each other and take generators of $a, b \in O_{K}$ of $\mathfrak{a}$ and $\mathfrak{b}$, respectively. Then, as rational functions on $E$, we have
\[
\frac{\theta_{\mathfrak{b}}^{N\mathfrak{a}}}{\theta_{\mathfrak{b}} \circ [a]} = \frac{\theta_{\mathfrak{a}}^{N\mathfrak{b}}}{\theta_{\mathfrak{a}} \circ [b]}.
\]
\end{lemma}
\begin{proof} 
Since the divisors associated to both sides coincide, the two functions are equal up to a multiple of $K^{\times}$. Moreover, for an arbitrary primitive $\mathfrak{m}$-torsion point $\tau$ of $E$ where $\mathfrak{m}$ is prime to $\mathfrak{a}\mathfrak{b}$, 
\[
\theta_{\mathfrak{a}}(b\tau)\theta_{\mathfrak{b}}(\tau)^{N\mathfrak{a}}
=
\theta_{\mathfrak{b}}(a\tau)\theta_{\mathfrak{a}}(\tau)^{N\mathfrak{b}}
\] holds by Proposition \ref{elliptic_unit}. This concludes the proof.
\end{proof}

By using Proposition \ref{distribution_relation} and Proposition \ref{elliptic_unit}, one can compute norms of special values of $\theta_{\mathfrak{a}}$ as follows.

\begin{proposition}[{\cite[Chapter $\rm{I\hspace{-.01em}I}$ \textsection 2, 2.5 Proposition $(\mathrm{i})$]{dS}}]\label{elliptic_norm}
Let $\mathfrak{f}$ be a nontrivial ideal of $O_{K}$, $\mathfrak{l}=(l)$ a maximal ideal of $O_{K}$, $\mathfrak{a}$ a nontrivial ideal of $O_{K}$ relatively prime to $\mathfrak{g} \coloneqq \mathfrak{f}\mathfrak{l}$ and $\tau$ a primitive $\mathfrak{g}$-torsion point of $E$. If we denote the number of roots of unity in $K$ congruent to $1 \bmod \mathfrak{f}$ (resp. $\mathfrak{g}$) by $w_{\mathfrak{f}}$ (resp. $w_{\mathfrak{g}}$) and set $e \coloneqq \frac{w_{\mathfrak{f}}}{w_{\mathfrak{g}}}$, we have
\begin{eqnarray*}
N_{K(\mathfrak{g})/K(\mathfrak{f})}(\theta_{\mathfrak{a}}(\tau)^{e}) = 
\begin{cases}
\theta_{\mathfrak{a}}(l\tau)^{1-\sigma_{\mathfrak{l}}^{-1}}  & (\text{if } \mathfrak{l} \nmid \mathfrak{f}) \\
\theta_{\mathfrak{a}}(l\tau) & (\text{if } \mathfrak{l} \mid \mathfrak{f})
\end{cases}.
\end{eqnarray*}
\end{proposition}


\subsection{The group of elliptic units}\label{2.2}

In this subsection, we briefly recall Rubin's group of elliptic units. A useful reference is given by Schmitt \cite{Sch}, which carefully compares the three different kinds of elliptic units (defined by Rubin, de Shalit and Yager, respectively).

\begin{definition}[{\cite[Definition 2.1]{Sch}}, see also {\cite[\textsection 1]{Ru2}}]\label{elliptic_unit_R}
Let $F$ be a finite abelian extension of $K$.
\begin{enumerate}
\item We define a subgroup $C_{F}$ of $O_{F}^{\times}$ to be a subgroup generated by 
\[
N_{FK(\mathfrak{m})/F}(\theta_{\mathfrak{a}}(\tau))^{\sigma-1},
\] where $\mathfrak{m}$ ranges over the ideals of $O_{K}$ such that $O_{K}^{\times} \to O_{K}/\mathfrak{m}$ is injective, $\tau$ over the primitive $\mathfrak{m}$-torsion points of $E$, $\mathfrak{a}$ over the ideals of $O_{K}$ with $(\mathfrak{a}, 6\mathfrak{m})=1$ and $\sigma$ over $\mathrm{Gal}(F/K)$.  Note that $C_{F}$ is a  $\mathrm{Gal}(F/K)$-submodule of $O_{F}^{\times}$.
\item We define a subgroup $C(F)$ of $O_{F}^{\times}$, \emph{the group of elliptic units of $F$}, by
\[
C(F) \coloneqq \mu(F)C_{F}.
\]
\item We define a group $\mathcal{C}$ by the projective limit
\[
\mathcal{C} \coloneqq \varprojlim_{n} C(K(p^{n})) \otimes \mathbb{Z}_{p},
\] where the transition map $C(K(p^{n+1})) \otimes \mathbb{Z}_{p} \to C(K(p^{n})) \otimes \mathbb{Z}_{p} $ is induced by the norm map from $K(p^{n+1})^{\times}$ to $K(p^{n})^{\times}$ for every $n$.
\end{enumerate}
\end{definition}

\begin{remark}
In fact, the definitions appearing in \cite{Sch} \cite{Ru2} are slightly different from the definitions given above, since in \cite{Sch} \cite{Ru2} they  consider suitable $12$-th roots of elliptic units. However, since we assume $p \geq 5$, the resulting groups become the same upon taking tensor products with $\mathbb{Z}_{p}$.
\end{remark}

The following theorem relates the group of elliptic units $C(F)$ with the $p$-part of the class number of $F$. This theorem is essential in Section \ref{5}.

\begin{theorem}[{\cite[Theorem 1.3]{Ru2}}]\label{Rubin_cl}
Let $F$ be an abelian extension of $K$, and assume that the degree $[F \colon K]$ is prime to $p$. Then the $p$-part of the index $[O_{F}^{\times} \colon C(F)]$ is equal to the $p$-part of the class number of $F$.
\end{theorem}

First, we prove that one can impose more conditions on $\mathfrak{m}$ appearing in the definition of the group of elliptic units:

\begin{lemma}\label{useful_lemma}
Let $n \geq 1$ be an integer. The group $C_{K(p^{n})}$ is generated by 
\[
\theta_{\mathfrak{a}}(\tau)^{\sigma-1},
\] where $\mathfrak{m}$ ranges over ideals of the form $\p^{m_{1}}\bar{\p}^{m_{2}}$ for some $(m_{1}, m_{2}) \in \mathbb{Z}^{2}_{\geq 0} \setminus \{ (0,0) \}$ such that $m_{i} \leq n$ $(i=1,2)$, $\tau$ over the primitive $\mathfrak{m}$-torsion points of $E$, $\mathfrak{a}$ over the ideals of $O_{K}$ with $(\mathfrak{a}, 6\mathfrak{m})=1$ and $\sigma$ over $\mathrm{Gal}(K(p^{n})/K)$.
\end{lemma}

\begin{proof}
Let $\mathfrak{n}$ be an arbitrary ideal of $O_{K}$ such that $O_{K}^{\times} \to O_{K}/\mathfrak{n}$ is injective. Write $\mathfrak{n} = \p^{m_{1}}\bar{\p}^{m_{2}} \mathfrak{m}$ for an ideal $\mathfrak{m}$ with $(\mathfrak{m},p)=1$. Let $\tau$ be a primitive $\mathfrak{n}$-torsion point of $E$ and $\mathfrak{a}$ an arbitrary ideal with $(\mathfrak{a}, 6\mathfrak{n})=1$.  

First, by the equality
\[
K(p^{n}) \cap K(\mathfrak{n})= K(\p^{\min(n, m_{1})}\bar{\p}^{\min(n, m_{2})}),
\] there is a natural isomorphism
\[
\mathrm{Gal}(K(p^{n})K(\mathfrak{n})/K(p^{n})) \xrightarrow{\sim} \mathrm{Gal}(K(\mathfrak{n})/K(\p^{\min(n, m_{1})}\bar{\p}^{\min(n, m_{2})})).
\] Since $\theta_{\mathfrak{a}}(\tau)$ is contained in $K(\mathfrak{n})$, we have
\[
N_{K(p^{n})K(\mathfrak{n})/K(p^{n})} \theta_{\mathfrak{a}}(\tau)
= N_{K(\mathfrak{n})/K(\p^{\min(n, m_{1})}\bar{\p}^{\min(n, m_{2})})} \theta_{\mathfrak{a}}(\tau). \tag{$\ast$}
\] 
If $(m_{1}, m_{2})=(0,0)$, then ($\ast$) is contained in $K$ and taking the action of $\sigma-1$ for $\sigma \in \mathrm{Gal}(K(p^{n})/K)$ yields $1$.

Hence we may assume  $(m_{1}, m_{2}) \neq (0,0)$. By Proposition \ref{elliptic_norm}, ($\ast$) is contained in a Galois submodule generated by $\theta_{\mathfrak{a}}(\pi^{m_{1}-\min(n, m_{1})}\bar{\pi}^{m_{2}-\min(n,m_{2})} \alpha \tau)$ where $\alpha$ is a generator of $\mathfrak{m}$. The claim follows since $\pi^{m_{1}-\min(n, m_{1})}\bar{\pi}^{m_{2}-\min(n,m_{2})} \alpha \tau$ is a primitive $\p^{\min(n, m_{1})}\bar{\p}^{\min(n, m_{2})}$-torsion point of $E$. 
\end{proof}

Recall that we fixed a basis $\{\omega_{1}, \omega_{2} \}$ of $T_{p}(E)$ such that $\omega_{1}$ (resp. $\omega_{2}$) generates $T_{\p}(E)$ (resp. $T_{\bar{\p}}(E)$).

\begin{proposition}\label{useful_prop}
Let $n \geq 1$ be an integer. The group $C_{K(p^{n})}$ is generated by 
\[
\theta_{\mathfrak{a}}(\omega_{1,n}+\omega_{2,n})^{\sigma-1}, \theta_{\mathfrak{a}}(\omega_{1,n})^{\sigma-1} \quad \text{and} \quad \theta_{\mathfrak{a}}(\omega_{2,n})^{\sigma-1}
\] where $\sigma$ ranges over $\mathrm{Gal}(K(p^{n})/K)$ and $\mathfrak{a}$ over the ideals of $O_{K}$ with $(\mathfrak{a}, 6p)=1$, $(\mathfrak{a}, 6\p)=1$ and $(\mathfrak{a}, 6\bar{\p})=1$, respectively.
\end{proposition}
\begin{proof}

We consider an arbitrary element of the form $\theta_{\mathfrak{a}}(\tau)$, where $\mathfrak{a}$ is an ideal of $O_{K}$ with $(\mathfrak{a}, 6\p^{m_{1}}\bar{\p}^{m_{2}})=1$ and $\tau$ is a primitive $\p^{m_{1}}\bar{\p}^{m_{2}}$-torsion point of $E$ for a pair of nonnegative integers $(m_{1}, m_{2})$ with $m_{i} \leq n$ for $i=1,2$. By Lemma \ref{useful_lemma}, it suffices to prove that $\theta_{\mathfrak{a}}(\tau)$ is contained in a $\mathrm{Gal}(K(p^{n})/K)$-submodule of $K(p^{n})^{\times}$ generated by $\theta_{\mathfrak{a}}(\omega_{1,n}+\omega_{2,n})$ if $m_{1}m_{2} \neq 0$, by $\theta_{\mathfrak{a}}(\omega_{2,n})$ if $m_{1}=0$ and by $\theta_{\mathfrak{a}}(\omega_{1,n})$ if $m_{2}=0$.

First we assume $m_{1}=0$. Then Proposition \ref{elliptic_unit} (2), we can find an element $\sigma \in \mathrm{Gal}(K(\bar{\p}^{m_{2}})/K)$ such that
\[
  \theta_{\mathfrak{a}}(\tau)=\sigma \left(\theta_{\mathfrak{a}}(\omega_{2,m_{2}}) \right)
\] holds. Moreover, by Proposition \ref{elliptic_norm}, $\theta_{\mathfrak{a}}(\omega_{2,m_{2}})$ can be written as the norm image of $\theta_{\mathfrak{a}}(\omega_{2,n})$. This concludes the proof and the same argument works when $m_{2}=0$. If $m_{1}m_{2} \neq 0$, then by Proposition \ref{elliptic_unit} (2) again, there is $\sigma \in \mathrm{Gal}(K(\p^{m_{1}}\bar{\p}^{m_{2}})/K)$ such that 
\[
  \theta_{\mathfrak{a}}(\tau)=\sigma \left(\theta_{\mathfrak{a}}(\omega_{2,m_{1}}+\omega_{2,m_{2}}) \right)
\] and Proposition \ref{elliptic_norm} imply that $\theta_{\mathfrak{a}}(\omega_{2,m_{1}}+\omega_{2,m_{2}})$ can be written as the norm image of $\theta_{\mathfrak{a}}(\omega_{1,n}+\omega_{2,n})$ as desired. 
\end{proof}

\begin{remark}\label{useful_remark}
By a similar argument, it follows that the group $C_{K(\p^{n})}$ (resp. $C_{K(\bar{\p}^{n})}$) is generated by $\theta_{\mathfrak{a}}(\omega_{1,n})^{\sigma-1}$ (resp. $\theta_{\mathfrak{a}}(\omega_{2,n})^{\sigma-1}$), where $\mathfrak{a}$ runs through ideals of $O_{K}$ with $(\mathfrak{a}, 6\p)=1$ (resp. $(\mathfrak{a}, 6\bar{\p})=1$) and $\sigma$ through $\mathrm{Gal}(K(\p^{n})/K)$ (resp. $\mathrm{Gal}(K(\bar{\p}^{n})/K)$) for every $n \geq 1$. 
\end{remark}

We separate the group of elliptic units into the following three subgroups:

\begin{definition}\label{C_D}
Let $n \geq 1$ be an integer.
\begin{enumerate}
\item Let $C'(K(p^{n}))$ be the subgroup of $C(K(p^{n}))$ generated by $\mu(K(p^{n}))$ and  
\[
\theta_{\mathfrak{a}}(\omega_{1,n}+\omega_{2,n})^{\sigma-1}
\] where $\mathfrak{a}$ ranges over the ideals of $O_{K}$ with $(\mathfrak{a}, 6p)=1$ and $\sigma$ over $\mathrm{Gal}(K(p^{n})/K)$. 

\item Let $D_{1}(K(p^{n}))$ be the subgroup of $C(K(p^{n}))$ generated by 
\[
\theta_{\mathfrak{a}}(\omega_{1,n})^{\sigma-1}
\] where $\mathfrak{a}$ ranges over the ideals of $O_{K}$ with $(\mathfrak{a}, 6\p)=1$ and $\sigma$ over $\mathrm{Gal}(K(\p^{n})/K)$. Similarly, let $D_{2}(K(p^{n}))$ be the subgroup of $C(K(p^{n}))$ generated by 
\[
\theta_{\mathfrak{a}}(\omega_{2,n})^{\sigma-1}
\] where $\mathfrak{a}$ ranges over the ideals of $O_{K}$ with $(\mathfrak{a}, 6\bar{\p})=1$ and $\sigma$ over $\mathrm{Gal}(K(\bar{\p}^{n})/K)$.

\item We define three subgroups $\mathcal{C'}$, $\mathcal{D}_{1}$ and $\mathcal{D}_{2}$ of $\mathcal{C}$ by
\[
\mathcal{C'} \coloneqq \varprojlim_{n} C'(K(p^{n})) \otimes \mathbb{Z}_{p},
\]
\[
\mathcal{D}_{1} \coloneqq \varprojlim_{n} D_{1}(K(p^{n})) \otimes \mathbb{Z}_{p} \quad \text{and} \quad 
\mathcal{D}_{2} \coloneqq \varprojlim_{n} D_{2}(K(p^{n})) \otimes \mathbb{Z}_{p},
\]
 respectively.
\end{enumerate}
\end{definition}

Note that the group $\mathcal{C}$ is generated by $\mathcal{D}_{1}$, $\mathcal{D}_{2}$ and $\mathcal{C}'$ by Proposition \ref{useful_prop}. The following propositions shows $\mathcal{C}$ coincides with $\mathcal{C}'$:

\begin{proposition}\label{R_dS}
  The subgroups $\mathcal{D}_{1}$ and $\mathcal{D}_{2}$ are trivial. In particular, $\mathcal{C}=\mathcal{C}'$.
\end{proposition}

\begin{proof}
Let $(x_{n})_{n \geq 1} \in \mathcal{D}_{1}$ and fix an arbitrary integer $n_{0} \geq 1$. Since $x_{n_{0}+N}$ is contained in $K(\p^{n_{0}+N}) \subset K(p^{n_{0}}\p^{N})$ for every $N \geq 1$, it holds that 
\begin{eqnarray*}
x_{n_{0}}
&=& N_{K(p^{n_{0}+N})/K(p^{n_{0}})}(x_{n_{0}+N}) \\
&=& N_{K(p^{n_{0}}\p^{N})/K(p^{n_{0}})}(N_{K(p^{n_{0}+N})/K(p^{n_{0}}\p^{N})}(x_{n_{0}+N})) \\
&=& N_{K(p^{n_{0}}\p^{N})/K(p^{n_{0}})}(x_{n_{0}+N}^{[K(p^{n_{0}+N}) \colon K(p^{n_{0}}\p^{N})]}) \in (D_{1}(K(p^{n_{0}})) \otimes \mathbb{Z}_{p})^{p^{N}}.
\end{eqnarray*} Since $N$ is arbitrary, we conclude $x_{n_{0}}=1$. This shows that $\mathcal{D}_{1}$ is trivial, and the same proof works for $\mathcal{D}_{2}$.
\end{proof}

\section{Rewriting elliptic Soul\'e characters}\label{3}

Our first task toward proving the nontriviality of the elliptic Soul\'e characters is to modify $\epsilon_{\boldsymbol{m},n}$ into a product of elliptic units in $K(p^{n})$. This can be done by multiplying $\epsilon_{\boldsymbol{m},n}$ by a suitable nonzero multiple of $\mathbb{Z}_{p}$. 

\medskip

In the following of this section, let $\boldsymbol{m}=(m_{1}, m_{2}) \in I$, $\mathfrak{a}$ a nontrivial ideal of $O_{K}$ prime to $p$ and $\alpha \in O_{K}$ a generator of $\mathfrak{a}$. Since the Artin symbol corresponding to $\mathfrak{a}$ acts on $p$-power torsion points of $E$ by multiplication by $\alpha$ up to multiplication by an element of $O_{K}^{\times}=\mathrm{Aut}(E)$, which leaves $\theta_{p}$ invariant, it follows that
\[
\sigma_{\mathfrak{a}} \left( \theta_{p}(\tau) \right)=\theta_{p}(\alpha \tau)
\] for every nontrivial $p$-power torsion point $\tau$ of $E$.

\begin{lemma}\label{mathfrak_a}
Let $n \geq 1$ be an integer. As an element of $K(p^{\infty})^{\times}/K(p^{\infty})^{\times p^{n}}$, we have the following equality. 
\[
\epsilon_{\boldsymbol{m}, n}^{N\mathfrak{a}-\chi^{\boldsymbol{1}-\boldsymbol{m}}(\sigma_{\mathfrak{a}})}=\prod_{\substack{0 \leq a,b <p^{n} \\ p \nmid \gcd(a,b)}}\theta_{\mathfrak{a}}(a\omega_{1,n}+b\omega_{2,n})^{a^{m_{1}-1}b^{m_{2}-1}}.
\]
\end{lemma}
\begin{proof} 
By the above discussion, it holds that 
\begin{eqnarray*}
\theta_{p}(a\omega_{1,n+1+k}+b\omega_{2,n+1+k})
&=&\theta_{p}\left(\sigma_{\mathfrak{a}}(\sigma_{\mathfrak{a}}^{-1}(a\omega_{1,n+1+k}+b\omega_{2,n+1+k}))\right) \\
&=& (\theta_{p}\circ[\alpha])(a\chi_{1}(\sigma_{\mathfrak{a}}^{-1})\omega_{1,n+1+k}+b\chi_{2}(\sigma_{\mathfrak{a}}^{-1})\omega_{2,n+1+k})
\end{eqnarray*} for every $k \geq 0$. Hence it follows that 
\[
\epsilon_{\boldsymbol{m}, n}^{\chi^{\boldsymbol{1}-\boldsymbol{m}}(\sigma_{\mathfrak{a}})}= \prod_{\substack{0 \leq a,b <p^{n} \\ p \nmid \gcd(a,b)}}\left( 
\prod_{k=0}^{\infty} (\theta_{p} \circ [\alpha])(a\omega_{1, n+1+k}+b\omega_{2,n+1+k})^{-p^{2k}}
\right)^{a^{m_{1}-1}b^{m_{2}-1}},
\] and we obtain
\[
\epsilon_{\boldsymbol{m}, n}^{-N\mathfrak{a}+\chi^{\boldsymbol{1}-\boldsymbol{m}}(\sigma_{\mathfrak{a}})}
=\prod_{\substack{0 \leq a,b <p^{n} \\ p \nmid \gcd(a,b)}}
\left(
\prod_{k=0}^{n} (\frac{(\theta_{p} \circ [\alpha])(a\omega_{1,n+1+k}+b\omega_{2,n+1+k})}{\theta_{p}(a\omega_{1,n+1+k}+b\omega_{2, n+1+k})^{N\mathfrak{a}}})^{-p^{2k}}
\right)^{a^{m_{1}-1}b^{m_{2}-1}}.
\] By Lemma \ref{p_to_a}, we simplify the right-hand side as follows.
\begin{eqnarray*}
&=& \prod_{\substack{ 0 \leq a,b <p^{n} \\ p \nmid \gcd(a,b)}}
\left(
\prod_{k=0}^{n} 
\left(
\frac{(\theta_{\mathfrak{a}} \circ [p])(a\omega_{1, n+1+k}+b\omega_{2, n+1+k})}{\theta_{\mathfrak{a}}(a\omega_{1, n+1+k}+b\omega_{2,n+1+k})^{p^{2}}}
\right)^{-p^{2k}}
\right)^{a^{m_{1}-1}b^{m_{2}-1}} \\
&=& \prod_{\substack{ 0 \leq a,b <p^{n} \\ p \mid \gcd(a,b)}}
\left(
\prod_{k=0}^{n} 
\left(
\frac{\theta_{\mathfrak{a}}(a\omega_{1,n+k}+b\omega_{2,n+k})}{\theta_{\mathfrak{a}}(a\omega_{1,n+1+k}+b\omega_{2,n+1+k})^{p^{2}}}
\right)^{-p^{2k}}
\right)^{a^{m_{1}-1}b^{m_{2}-1}} \\
&=&  \prod_{\substack{ 0 \leq a,b <p^{n} \\ p \nmid \gcd(a,b)}} \theta_{\mathfrak{a}} (a\omega_{1,n}+b\omega_{2,n})^{-a^{m_{1}-1}b^{m_{2}-1}}.
\end{eqnarray*} This concludes the proof.
\end{proof}

\subsection{The case where $\boldsymbol{m}=(m_{1}, m_{2}) \geq (2,2)$}\label{3.1}

In this subsection, we assume that $\boldsymbol{m}=(m_{1}, m_{2}) \geq (2,2)$. In the following, we modify $\epsilon_{\boldsymbol{m},n}$ into a certain form that only involves $\theta_{\mathfrak{a}}$ evaluated over primitive $p^{n}$-torsion points of $E$. 

\begin{lemma}\label{non_deg_case}
Let $n \geq 1$ be an integer. As an element of $K(p^{n})^{\times}/K(p^{n})^{\times p^{n}}$, 
\[
\prod_{\substack{a, b \in \mathbb{Z}/p^{n}\mathbb{Z} \\ p \nmid \gcd(a,b)}}\theta_{\mathfrak{a}}(a\omega_{1,n}+b\omega_{2,n})^{a^{m_{1}-1}b^{m_{2}-1}}
\] is equal to the $\left( \frac{1}{1-i^{\boldsymbol{m}-\boldsymbol{1}}(\pi)}+\frac{1}{1-i^{\boldsymbol{m}-\boldsymbol{1}}(\bar{\pi})}-1 \right)$-th power of
\[
\epsilon_{\boldsymbol{m}, n, \mathfrak{a}} \coloneqq \prod_{a, b \in (\mathbb{Z}/p^{n}\mathbb{Z})^{\times}}\theta_{\mathfrak{a}}(a\omega_{1,n}+b\omega_{2,n})^{a^{m_{1}-1}b^{m_{2}-1}}.
\]
\end{lemma}

\begin{proof}
First, note that
\begin{eqnarray*}
\epsilon_{\boldsymbol{m}, n,\mathfrak{a}}^{p^{m_{1}-1}}
&=&  \prod_{a, b \in (\mathbb{Z}/p^{n}\mathbb{Z})^{\times}}\theta_{\mathfrak{a}}(a\omega_{1, n}+b\omega_{2, n})^{(pa)^{m_{1}-1}b^{m_{2}-1}} \\
&=& \prod_{\substack{\bar{a} \in p(\mathbb{Z}/p^{n}\mathbb{Z})^{\times} \\ b \in (\mathbb{Z}/p^{n}\mathbb{Z})^{\times}}} \left( \prod_{\substack{a \in (\mathbb{Z}/p^{n}\mathbb{Z})^{\times} \\ pa \equiv \bar{a} \bmod p^{n}}}\theta_{\mathfrak{a}}(a\omega_{1, n}+b\omega_{2,n}) \right)^{\bar{a}^{m_{1}-1}b^{m_{2}-1}}.
\end{eqnarray*} On the other hand, the distribution relation (Proposition \ref{distribution_relation}) implies
\begin{eqnarray*}
\prod_{\substack{ a \in (\mathbb{Z}/p^{n}\mathbb{Z})^{\times} \\ pa \equiv \bar{a} \bmod p^{n}}}\theta_{\mathfrak{a}}(a\omega_{1,n}+b\omega_{2,n})
&=&\prod_{\tau \in E[\p]}\theta_{\mathfrak{a}} (a\omega_{1,n}+b\omega_{2,n}+\tau) = \theta_{\mathfrak{a}}(\pi(a\omega_{1,n}+b\omega_{2,n})) \\
&=&\theta_{\mathfrak{a}}(\bar{\pi}^{-1}\bar{a} \omega_{1,n}+\pi b\omega_{2,n}).
\end{eqnarray*}
Hence we conclude
\begin{eqnarray*}
\epsilon_{\boldsymbol{m}, n,\mathfrak{a}}^{p^{m_{1}-1}} &=&
\prod_{ a,b \in (\mathbb{Z}/p^{n}\mathbb{Z})^{\times}} \theta_{\mathfrak{a}}(a\omega_{1,n}+b\omega_{2,n})^{(pa)^{m_{1}-1}b^{m_{2}-1}} \\
&=&  \prod_{\substack{\bar{a} \in p(\mathbb{Z}/p^{n}\mathbb{Z})^{\times} \\ b \in (\mathbb{Z}/p^{n}\mathbb{Z})^{\times}}} \theta_{\mathfrak{a}}(\bar{\pi}^{-1}\bar{a}\omega_{1,n}+\pi b \omega_{2,n})^{\bar{a}^{m_{1}-1}b^{m_{2}-1}} \\
&=& \left( \prod_{\substack{\bar{a} \in p(\mathbb{Z}/p^{n}\mathbb{Z})^{\times}\\ b \in (\mathbb{Z}/p^{n}\mathbb{Z})^{\times}}} \theta_{\mathfrak{a}}(\bar{a}\omega_{1,n}+b\omega_{2,n})^{\bar{a}^{m_{1}-1}b^{m_{2}-1}} \right)^{i_{\p}^{m_{1}-1}(\bar{\pi})i_{\bar{\p}}^{m_{2}-1}(\pi^{-1})}. 
\end{eqnarray*} Since $p^{m_{1}-1}=i_{1}^{m_{1}-1}(p)=i_{1}^{m_{1}-1}(\pi)i_{1}^{m_{1}-1}(\bar{\pi})$, we have
\[
\epsilon_{\boldsymbol{m},n,\mathfrak{a}}^{i^{\boldsymbol{m}-1}(\pi)} 
=
\prod_{\substack{a \in p(\mathbb{Z}/p^{n}\mathbb{Z})^{\times} \\ b \in (\mathbb{Z}/p^{n}\mathbb{Z})^{\times}}} \theta_{\mathfrak{a}}(a\omega_{1,n}+b\omega_{2,n})^{a^{m_{1}-1}b^{m_{2}-1}}.
\] 
For every $1 \leq i < n$, the same argument as above shows
\[
\epsilon_{\boldsymbol{m} ,n,\mathfrak{a}}^{i^{\boldsymbol{m}-\boldsymbol{1}}(\pi)^{i}}
=
\prod_{\substack{a \in p^{i}(\mathbb{Z}/p^{n}\mathbb{Z})^{\times} \\ b \in (\mathbb{Z}/p^{n}\mathbb{Z})^{\times}}} \theta_{\mathfrak{a}}(a\omega_{1,n}+b\omega_{2,n})^{a^{m_{1}-1}b^{m_{2}-1}} 
\] and 
\[
\epsilon_{\boldsymbol{m},n,\mathfrak{a}}^{i^{\boldsymbol{m}-\boldsymbol{1}}(\bar{\pi})^{i}}
=
\prod_{\substack{a \in (\mathbb{Z}/p^{n}\mathbb{Z})^{\times} \\ b \in p^{i}(\mathbb{Z}/p^{n}\mathbb{Z})^{\times}}} \theta_{\mathfrak{a}}(a\omega_{1,n}+b\omega_{2,n})^{a^{m_{1}-1}b^{m_{2}-1}}
\] The assertion follows from these equalities.
\end{proof}

\subsection{The case where $m_{1}=1$ or $m_{2}=1$}\label{3.2}

In this subsection, we assume that $m_{1}=1$ or $m_{2}=1$. We only treat the case where $m_{2}=1$, since the remaining case can be treated in the same manner. In the following, we reduce to a one-variable situation by considering the primitive $\p^{n}$-torsion points instead of the primitive $p^{n}$-torsion points.

\begin{lemma}\label{deg_case}
Let $n \geq 1$ be an integer. As an element of $K(p^{n})^{\times}/K(p^{n})^{\times p^{n}}$, 
\[
\prod_{\substack{a, b \in \mathbb{Z}/p^{n}\mathbb{Z}  \\ p \nmid \gcd(a,b)}}\theta_{\mathfrak{a}}(a\omega_{1, n}+b\omega_{2,n})^{a^{m_{1}-1}}
\] is the $\left( \frac{1-i^{\boldsymbol{m}-\boldsymbol{1}}(\bar{\pi})}{1-i^{\boldsymbol{m}-\boldsymbol{1}}(\pi)}+i^{\boldsymbol{m}-\boldsymbol{1}}(\bar{\pi}) \right)$-th power of
\[
\epsilon_{\boldsymbol{m},n, \mathfrak{a}} \coloneqq \prod_{a \in (\mathbb{Z}/p^{n}\mathbb{Z})^{\times}} \theta_{\mathfrak{a}}(a\bar{\pi}^{n}\omega_{1,n})^{a^{m_{1}-1}}.
\]
\end{lemma}

\begin{proof} First, we divide the set $\{ (a, b) \in (\mathbb{Z}/p^{n}\mathbb{Z})^{2} \mid p \nmid \gcd(a,b) \}$ into the following two subsets
\[
S_{1} \coloneqq \{ (a, b) \in (\mathbb{Z}/p^{n}\mathbb{Z})^{2}  \mid b \in (\mathbb{Z}/p^{n}\mathbb{Z})^{\times} \} 
\] and
\[
S_{2} \coloneqq \{ (a, b) \in (\mathbb{Z}/p^{n}\mathbb{Z})^{2}  \mid a \in (\mathbb{Z}/p^{n}\mathbb{Z})^{\times} \text{ and } b \equiv 0 \bmod p  \}.
\]  Then, the same argument as in the proof of Lemma \ref{non_deg_case} shows the  equality
\[
\prod_{(a,b) \in S_{1}}\theta_{\mathfrak{a}}(a\omega_{1,n}+b\omega_{2,n})^{a^{m_{1}-1}}=\left( \prod_{ a , b \in (\mathbb{Z}/p^{n}\mathbb{Z})^{\times}}\theta_{\mathfrak{a}}(a\omega_{1,n}+b\omega_{2, n})^{a^{m_{1}-1}} \right)^{\sum_{i=0}^{n-1}i^{\boldsymbol{m}-\boldsymbol{1}}(\pi)^{i}}.
\] 
To compute the right-hand side, note that 
\begin{eqnarray*}
\prod_{ a , b \in (\mathbb{Z}/p^{n}\mathbb{Z})^{\times}}\theta_{\mathfrak{a}}(a\omega_{1,n}+b\omega_{2,n})^{a^{m_{1}-1}} 
&=& \prod_{ a \in (\mathbb{Z}/p^{n}\mathbb{Z})^{\times}}\prod_{ b \in (\mathbb{Z}/p^{n}\mathbb{Z})^{\times}}\theta_{\mathfrak{a}}(a\omega_{1,n}+b\omega_{2,n})^{a^{m_{1}-1}} \\
&=& \prod_{ a \in (\mathbb{Z}/p^{n}\mathbb{Z})^{\times}}N_{K(p^{n})/K(\p^{n})}(\theta_{\mathfrak{a}}(a\omega_{1,n}+\omega_{2,n}))^{a^{m_{1}-1}} \\
&=& \left( \prod_{ a \in (\mathbb{Z}/p^{n}\mathbb{Z})^{\times}} \theta_{\mathfrak{a}}(a\bar{\pi}^{n}\omega_{1,n})^{a^{m_{1}-1}} \right)^{1-\sigma_{\bar{\p}}^{-1}}.
\end{eqnarray*} Here, we use Lemma \ref{elliptic_norm} to deduce the last equality. Since the inverse of the Artin symbol $\sigma_{\bar{\p}}^{-1}$ sends $\theta_{\mathfrak{a}}(a\bar{\pi}^{n}\omega_{1, n})$ to $\theta_{\mathfrak{a}}(a\bar{\pi}^{n-1}\omega_{1,n})$, it follows that
\[
\left( \prod_{ a \in (\mathbb{Z}/p^{n}\mathbb{Z})^{\times}} \theta_{\mathfrak{a}}(a\bar{\pi}^{n}\omega_{1,n})^{a^{m_{1}-1}} \right)^{1-\sigma_{\bar{\p}}^{-1}}=\left( \prod_{ a \in (\mathbb{Z}/p^{n}\mathbb{Z})^{\times}} \theta_{\mathfrak{a}}(a\bar{\pi}^{n}\omega_{1,n})^{a^{m_{1}-1}} \right)^{1-i^{\boldsymbol{m}-\boldsymbol{1}}(\bar{\pi})}.
\] 

Next, we compute the product $\prod_{(a,b) \in S_{2}}\theta_{\mathfrak{a}}(a\omega_{1,n}+b\omega_{2,n})^{a^{m_{1}-1}}
$ by using Proposition \ref{distribution_relation} as follows:
\begin{eqnarray*}
\prod_{(a,b) \in S_{2}}\theta_{\mathfrak{a}}(a\omega_{1,n}+b\omega_{2,n})^{a^{m_{1}-1}}
&=& \prod_{ a \in (\mathbb{Z}/p^{n}\mathbb{Z})^{\times}} \left( \prod_{b \in p\mathbb{Z}/p^{n}\mathbb{Z}}\theta_{\mathfrak{a}}(a\omega_{1,n}+b\omega_{2,n}) \right)^{a^{m_{1}-1}} \\
&=& \prod_{ a \in (\mathbb{Z}/p^{n}\mathbb{Z})^{\times}} \left( \prod_{\tau \in E[\bar{\p}^{n-1}]}\theta_{\mathfrak{a}}(a\omega_{1,n}+\tau) \right)^{a^{m_{1}-1}} \\
&=& \prod_{ a \in (\mathbb{Z}/p^{n}\mathbb{Z})^{\times}} \theta_{\mathfrak{a}}(a\bar{\pi}^{n-1}\omega_{1,n})^{a^{m_{1}-1}} \\
&=& \left( \prod_{ a \in (\mathbb{Z}/p^{n}\mathbb{Z})^{\times}} \theta_{\mathfrak{a}}(a\bar{\pi}^{n}\omega_{1,n})^{a^{m_{1}-1}} \right)^{i^{\boldsymbol{m}-\boldsymbol{1}}(\bar{\pi})}.
\end{eqnarray*} 

Combining these computations yields that $
\prod_{\substack{a, b \in \mathbb{Z}/p^{n}\mathbb{Z}  \\ p \nmid \gcd(a,b)}}\theta_{\mathfrak{a}}(a\omega_{1,n}+b\omega_{2,n})^{a^{m_{1}-1}}$ multiplied by
\[
(1-i^{\boldsymbol{m}-\boldsymbol{1}}(\bar{\pi}))\sum_{i=0}^{\infty}i^{\boldsymbol{m}-\boldsymbol{1}}(\pi)^{i}+i^{\boldsymbol{m}-\boldsymbol{1}}(\bar{\pi})=\frac{1-i^{\boldsymbol{m}-\boldsymbol{1}}(\bar{\pi})}{1-i^{\boldsymbol{m}-\boldsymbol{1}}(\pi)}+i^{\boldsymbol{m}-\boldsymbol{1}}(\bar{\pi})
\] is equal to $\epsilon_{\boldsymbol{m}, n,\mathfrak{a}}$ as desired. 
\end{proof}

A similar computation yields the following.

\begin{lemma}
Let $n \geq 1$ be an integer. As an element of $K(p^{n})^{\times}/K(p^{n})^{\times p^{n}}$, 
\[
\prod_{\substack{a, b \in \mathbb{Z}/p^{n}\mathbb{Z} \\ p \nmid \gcd(a,b)}}\theta_{\mathfrak{a}}(a\omega_{1,n}+b\omega_{2, n})^{b^{m_{2}-1}}
\] is equal to the $\left( \frac{1-i^{\boldsymbol{m}-\boldsymbol{1}}(\pi)}{1-i^{\boldsymbol{m}-\boldsymbol{1}}(\bar{\pi})}+i^{\boldsymbol{m}-\boldsymbol{1}}(\pi) \right)$-th power of
\[
\epsilon_{\boldsymbol{m}, n, \mathfrak{a}} \coloneqq \prod_{b \in (\mathbb{Z}/p^{n}\mathbb{Z})^{\times}} \theta_{\mathfrak{a}}(b\pi^{n}\omega_{2,n})^{b^{m_{2}-1}}.
\]
\end{lemma}

\section{Nontriviality of elliptic Soul\'e characters}\label{4}

In this section, we prove Theorem \ref{thm:main} (1) (=Theorem \ref{non_vanishing}) and study various situations which imply the finiteness of $H^{2}_{\et}$ appearing in the assumption of the theorem. Let
\[
I \coloneqq \{ \boldsymbol{m}=(m_{1}, m_{2}) \in \mathbb{Z}_{ \geq 1}^{2} \setminus \{ \boldsymbol{1} \} \mid m_{1} \equiv m_{2} \bmod w_{K} \}.
\]

First, we prepare the following lemma. 

\begin{lemma}\label{Soule_factor}
Let $\boldsymbol{m}=(m_{1},m_{2}) \in \mathbb{Z}_{\geq 1}^{2}\setminus \{ \boldsymbol{1} \}$.
\begin{enumerate}
\item If $\boldsymbol{m} \not \in I$, then $\kappa_{\boldsymbol{m}}$ is trivial.
\item Let $\Omega_{p}$ be the maximal abelian pro-$p$ extension of $K(p^{\infty})$ unramified outside $p$. Then the character $\kappa_{\boldsymbol{m}}$ factors through the natural homomorphism \[
  G_{K(E[p^{\infty}])}^{\mathrm{ab}} \to \mathrm{Gal}(\Omega_{p}/K(p^{\infty})),
  \] which is surjective. 
\end{enumerate}
\end{lemma} 
\begin{proof}
(1) For every $\zeta \in O_{K}^{\times}$, it follows by observing the action of $O_{K}^{\times}$ on $E$ that $\kappa_{\boldsymbol{m}}=\zeta^{m_{1}-m_{2}}\kappa_{\boldsymbol{m}}$. Hence $\kappa_{\boldsymbol{m}}$ is trivial unless $\boldsymbol{m} \not \in I$.

(2) Since the intersection $\Omega_{p} \cap K(E[p^{\infty}])$ is an abelian extension of $K$ unramified outside $p$, we have $\Omega_{p} \cap K(E[p^{\infty}])=K(p^{\infty})$. The surjectivity follows from this equality. The rest of the assertion that the character $\kappa_{\boldsymbol{m}}$ factors through the homomorphism \[
  G_{K(E[p^{\infty}])}^{\mathrm{ab}} \to \mathrm{Gal}(\Omega_{p}/K(p^{\infty})),
\]follows by the definition of $\kappa_{\boldsymbol{m}}$ (Definition \ref{ellSoule}) and the fact $\theta_{p}(\omega)$ is a $p$-unit of $K(p^{\infty})$ for every $\omega \in  \left( E[p^{\infty}]\setminus E[p] \right)(\overline{\mathbb{Q}})$ cf. Remark \ref{pthpower}. 
\end{proof}

Lemma \ref{Soule_factor}(2) allows us to regard $\kappa_{\boldsymbol{m}}$ as an element of the cohomology group \[
  H^{1}_{\et}(O_{K(p^{\infty}), S_{p}}, \mathbb{Z}_{p}(\boldsymbol{m}))^{\mathrm{Gal}(K(p^{\infty})/K)}.
\]

\medskip

In the following of this paper, we abbreviate $O_{K(\p^{m_{1}}\bar{\p}^{m_{2}})}^{\times}$ as $E_{\boldsymbol{m}}$ for every $\boldsymbol{m}=(m_{1}, m_{2}) \in \mathbb{Z}^{2}_{\geq 0}$. If $m_{1}=m_{2}=m$, it is even abbreviated as $E_{m}$.

\begin{definition}\label{norm_unit}
We define a $\mathbb{Z}_{p}[[\mathrm{Gal}(K(p^{\infty})/K)]]$-module $\mathcal{E}$ to be 
\[
\mathcal{E} \coloneqq \varprojlim_{n} E_{n} \otimes \mathbb{Z}_{p},
\] where transition maps are taken to be norms. Note that this is equivalent to defining $\mathcal{E}$ to be $\varprojlim_{n} E_{n}/E_{n}^{p^{n}}$.
\end{definition}

\subsection{Conditional nontriviality of elliptic Soul\'e characters}\label{4.1}

In this section, we prove Theorem \ref{thm:main} (1) (=Theorem \ref{non_vanishing}). We adopt some techniques developed by Kings \cite{Ki3}  to study the Tamagawa number conjecture for CM elliptic curves. There he considers the \'etale (or Galois) cohomology groups with coefficients in the Tate twist of the $p$-adic Tate module $T_{p}(E)(k+1)=\mathbb{Z}_{p}(k+2, k+1) \oplus \mathbb{Z}_{p}(k+1, k+2)$ for $k \geq 0$.

First, we review the construction of the Soul\'e elliptic elements \cite[2.2.1]{Ki3}, which originates from Soul\'e \cite{So1} \cite{So2}. We fix $\boldsymbol{m}=(m_{1}, m_{2}) \in I$ and let $e=(e_{n})_{n \geq 1} \in \mathcal{E}$ be a norm compatible system of units. For every $n \geq 1$, the element $e_{n}$ can be regarded as an element of $H^{1}_{\et}(O_{K(p^{n}),S_{p}}, \mathbb{Z}/p^{n}(\boldsymbol{m}))$ via the inclusion from Kummer theory
\[
E_{n}/E_{n}^{p^{n}}(\boldsymbol{m}-\boldsymbol{1}) \subset H^{1}_{\et}(O_{K(p^{n}),S_{p}}, \mathbb{Z}/p^{n}(\boldsymbol{m})).
\] Moreover, by composing this inclusion with the corestriction map with respect to $\mathrm{Spec}(O_{K(p^{n}),S_{p}}) \to \mathrm{Spec}(O_{K,S_{p}})$, we obtain an element of $H^{1}_{\et}(O_{K, S_{p}},  \mathbb{Z}/p^{n}(\boldsymbol{m}))$. 

Then it holds by direct computation that such a system of elements forms a projective system with respect to the reduction map, hence defines an element of $H^{1}_{\et}(O_{K, S_{p}},  \mathbb{Z}_{p}(\boldsymbol{m}))$. This construction gives a homomorphism
\[
\mathcal{E}(\boldsymbol{m}-\boldsymbol{1})=(\varprojlim E_{n}/E_{n}^{p^{n}})(\boldsymbol{m}-\boldsymbol{1}) \to H^{1}_{\et}(O_{K, S_{p}},  \mathbb{Z}_{p}(\boldsymbol{m})),
\] and it follows that this homomorphism factors through the $\mathrm{Gal}(K(p^{\infty})/K)$-coinvariants of $\mathcal{E}(\boldsymbol{m}-\boldsymbol{1})$. We write the resulting homomorphism as
\[
e_{\boldsymbol{m}} \colon \mathcal{E}(\boldsymbol{m}-\boldsymbol{1})_{\mathrm{Gal}(K(p^{\infty})/K)}  \to H^{1}_{\et}(O_{K, S_{p}},  \mathbb{Z}_{p}(\boldsymbol{m})).
\]  Moreover, we use the same letter $e_{\boldsymbol{m}}$ to denote the homomorphism
\[
  \mathcal{E}(\boldsymbol{m}-\boldsymbol{1})_{\mathrm{Gal}(K(p^{\infty})/K)}  \to \mathrm{Hom}_{\mathrm{Gal}(K(p^{\infty})/K)}(\mathrm{Gal}(\Omega_{p}/K(p^{\infty})), \mathbb{Z}_{p}(\boldsymbol{m}))
\] obtained by composing $e_{\boldsymbol{m}}$ with the restriction map 
\[
  H^{1}_{\et}(O_{K, S_{p}},  \mathbb{Z}_{p}(\boldsymbol{m})) \to H^{1}_{\et}(O_{K(p^{\infty}), S_{p}}, \mathbb{Z}_{p}(\boldsymbol{m}))^{\mathrm{Gal}(K(p^{\infty})/K)}.
\] 
We remark that the both kernel and the cokernel of this restriction map are finite since both $H^{1}(\mathrm{Gal}(K(p^{\infty})/K), \mathbb{Z}_{p}(\boldsymbol{m}))$ and $H^{2}(\mathrm{Gal}(K(p^{\infty})/K), \mathbb{Z}_{p}(\boldsymbol{m}))$ are finite. 

\medskip

The next proposition allows us to relate $e_{\boldsymbol{m}}$ with the elliptic Soul\'e characters.

\begin{lemma}
For $u=(u_{n})_{n \geq 1} \in \mathcal{E}$, the character $e_{\boldsymbol{m}}(u) \colon \mathrm{Gal}(\Omega_{p}/K(p^{\infty})) \to \mathbb{Z}_{p}(\boldsymbol{m})$ modulo $p^{n}$ coincides with the Kummer character associated to the $p^{n}$-th root of 
\[
\prod_{\sigma \in \mathrm{Gal}(K(p^{n})/K)} \sigma(u_{n})^{\chi^{\boldsymbol{m}-\boldsymbol{1}}(\sigma)} .
\]
\end{lemma}
\begin{proof}
The assertion follows directly from the construction of $e_{\boldsymbol{m}}$.
\end{proof}

Note that the system of elliptic units $\left( \theta_{\mathfrak{a}}(\omega_{1,n}+\omega_{2,n}) \right)_{n \geq 1}$ is contained in $\mathcal{E}$ for every nontrivial ideal $\mathfrak{a}$ of $O_{K}$ relatively prime to $p$ by Lemma \ref{elliptic_norm}. 

\begin{definition}\label{dfn:kappa_a}
Let $\boldsymbol{m} \in I$ and $\mathfrak{a}$ a nontrivial ideal of $O_{K}$ relatively prime to $p$. We define a character $\kappa_{\boldsymbol{m},\mathfrak{a}} \colon \mathrm{Gal}(\Omega_{p}/K(p^{\infty})) \to \mathbb{Z}_{p}(\boldsymbol{m})$ to be the image of the system $\left( \theta_{\mathfrak{a}}(\omega_{1,n}+\omega_{2,n}) \right)_{n \geq 1}$ under the homomorphism $e_{\boldsymbol{m}}$.
\end{definition}

Then, results in the previous section yield the following formula:

\begin{proposition}\label{kappa_a}
Let $\boldsymbol{m}=(m_{1}, m_{2}) \in I$ and $\mathfrak{a}$ a nontrivial ideal of $O_{K}$ relatively prime to $p$. Then the following relation holds:
\[
w_{K} \left( \frac{1}{1-i^{\boldsymbol{m}-\boldsymbol{1}}(\pi)}+\frac{1}{1-i^{\boldsymbol{m}-\boldsymbol{1}}(\bar{\pi})}-1 \right) \kappa_{\boldsymbol{m},\mathfrak{a}}
=
\left(
  N\mathfrak{a}-\chi^{\boldsymbol{1}-\boldsymbol{m}}(\sigma_{\mathfrak{a}})
\right)
\kappa_{\boldsymbol{m}}.
\]
\end{proposition}

\begin{proof}
If $\boldsymbol{m} \geq (2,2)$, the assertion follows from Lemma \ref{mathfrak_a}, Lemma \ref{non_deg_case} and the fact that we count the same value $w_{K}$ times. If $m_{2}=1$, the character $\chi^{\boldsymbol{m}-\boldsymbol{1}}(=\chi_{1}^{m_{1}-1})$ factors through $\mathrm{Gal}(K(\p^{\infty})/K)$. Hence $\kappa_{\boldsymbol{m}, \mathfrak{a}}$ modulo $p^{n}$ corresponds to
\[
\prod_{\sigma \in \mathrm{Gal}(K(p^{n})/K)} \sigma\left(\theta_{\mathfrak{a}}(\omega_{1,n}+\omega_{2,n})\right)^{\chi^{\boldsymbol{m}-\boldsymbol{1}}(\sigma)}=\prod_{\sigma \in \mathrm{Gal}(K(\p^{n})/K)} \sigma \left(N_{K(p^{n})/K(\p^{n})}(\theta_{\mathfrak{a}}(\omega_{1,n}+\omega_{2,n}))\right)^{\chi^{\boldsymbol{m}-\boldsymbol{1}}(\sigma)}
\] for every $n \geq 1$. By Lemma \ref{elliptic_norm}, the right-hand side can be simplified as
\[
\left( \prod_{\sigma \in \mathrm{Gal}(K(\p^{n})/K)} \sigma(\theta_{\mathfrak{a}}(\bar{\pi}^{n}\omega_{1,n}))^{\chi_{1}^{m_{1}-1}(\sigma)}\right)^{1-\sigma_{\bar{\p}}^{-1}}.
\] Hence the right-hand side is equal to the $w_{K}^{-1}(1-i^{\boldsymbol{m}-\boldsymbol{1}}(\bar{\pi}))$-th power of $\epsilon_{\boldsymbol{m}, n, \mathfrak{a}}$. By Lemma \ref{mathfrak_a} and Lemma \ref{deg_case}, we conclude that 
\begin{eqnarray*}
\left( N\mathfrak{a}-\chi^{\boldsymbol{1}-\boldsymbol{m}}(\sigma_{\mathfrak{a}}) \right) \kappa_{\boldsymbol{m}}=
\left( \frac{1-i^{\boldsymbol{m}-\boldsymbol{1}}(\bar{\pi})}{1-i^{\boldsymbol{m}-\boldsymbol{1}}(\pi)}+i^{\boldsymbol{m}-\boldsymbol{1}}(\bar{\pi})
\right)
\frac{w_{K}}{1-i^{\boldsymbol{m}-\boldsymbol{1}}(\bar{\pi})}\kappa_{\boldsymbol{m}, \mathfrak{a}}
\end{eqnarray*} as desired.
\end{proof}

To prove a conditional nontriviality of the elliptic Soul\'e characters, we need the following observation.

\begin{lemma}\label{non_vanishing_lemma}
The $\mathbb{Z}_{p}$-module $e_{\boldsymbol{m}}(\mathcal{C}(\boldsymbol{m}-\boldsymbol{1})_{\mathrm{Gal}(K(p^{\infty})/K)})$ is contained in a $\mathbb{Z}_{p}$-submodule of $H^{1}_{\et}(O_{K(p^{\infty}), S_{p}}, \mathbb{Z}_{p}(\boldsymbol{m}))^{\mathrm{Gal}(K(p^{\infty})/K)}$ generated by $\kappa_{\boldsymbol{m}}$.
\end{lemma}
\begin{proof}
By Proposition \ref{R_dS}, the group $\mathcal{C}$ coincides with $\mathcal{C'}$. Moreover, by \cite[Lemma 3.6]{Sch}, it holds that $\varprojlim_{n} \mu(K(p^{n})) \subset \mathcal{C}'$ is trivial. Hence to conclude the proof, it suffices to show that the image of an element of the form $\theta_{\mathfrak{a}}(\omega_{1,n}+\omega_{2,n})^{\sigma-1}$ (cf. Definition \ref{C_D}) in $H^{1}_{\et}(O_{K(p^{\infty}), S_{p}}, \mathbb{Z}/p^{n}(\boldsymbol{m}))$ is a multiple of $\kappa_{\boldsymbol{m}} \bmod p^{n}$ for every $n \geq 1$. This claim follows from the proof of Proposition \ref{kappa_a}.
\end{proof}

\begin{theorem}\label{non_vanishing}
Let $\boldsymbol{m} \in I$, and assume that $H^{2}_{\et}(O_{K,S_{p}}, \mathbb{Z}_{p}(\boldsymbol{m}))$ is finite. Then the $\boldsymbol{m}$-th elliptic Soul\'e character $\kappa_{\boldsymbol{m}}$ is nontrivial.
\end{theorem}

\begin{proof}
The proof of this theorem is similar to that of \cite[Proposition 5.2.5]{Ki3}. We divide the proof into several steps. 

\medskip

First, let $M \subset K(p)$ (resp. $L \subset K(p^{\infty})$) be the subfield of $K(p)$ (resp. $K(p^{\infty})$) which corresponds to the kernel of $\chi^{\boldsymbol{m}-\boldsymbol{1}}$ modulo $p$ (resp. the kernel of $\chi^{\boldsymbol{m}-\boldsymbol{1}}$). Then $L/M$ is a $\mathbb{Z}_{p}$-extension. Let $L_{n}/M$ denote the subextension of $L/M$ such that $[L_{n}:M]=p^{n}$ for every $n \geq 1$ and $\mathcal{E}_{L} \coloneqq \varprojlim_{n} O_{L_{n}}^{\times}/O_{L_{n}}^{\times p^{n}}$. Similar to the construction of $e_{\boldsymbol{m}}$, one can construct a homomorphism:
\[
e_{\boldsymbol{m}, L}: \mathcal{E}_{L}(\boldsymbol{m}-\boldsymbol{1})_{\mathrm{Gal}(L/K)} \to H^{1}_{\et}(O_{K, S_{p}}, \mathbb{Z}_{p}(\boldsymbol{m}))
\] so that the following diagram commutes:
\begin{eqnarray*}
\begin{tikzcd}
\mathcal{E}(\boldsymbol{m}-\boldsymbol{1})_{\mathrm{Gal}(K(p^{\infty})/K)} \arrow[r, "e_{\boldsymbol{m}}"]\arrow[d] & H^{1}_{\et}(O_{K, S_{p}}, \mathbb{Z}_{p}(\boldsymbol{m})) \arrow[d, equal] \\
\mathcal{E}_{L}(\boldsymbol{m}-\boldsymbol{1})_{\mathrm{Gal}(L/K)} \arrow[r, "e_{\boldsymbol{m}, L}"] & H^{1}_{\et}(O_{K, S_{p}}, \mathbb{Z}_{p}(\boldsymbol{m})). 
\end{tikzcd}
\end{eqnarray*} 
Here, the left vertical arrow is induced by norm maps $N_{K(p^{n})/L_{n}}$. 

\medskip

Now let $\mathcal{A}_{L}$ be the projective limit of the $p$-part of the ideal class group of $L_{n}$ for $n \geq 1$, where transition maps are induced by norms. By a similar argument as in \cite[Proposition 5.2.5]{Ki3}, it holds that $\mathcal{A}_{L}(\boldsymbol{m}-1)_{\mathrm{Gal}(L/K)}$ and the kernel of $e_{\boldsymbol{m}, L}$ are finite under the assumption that $H^{2}_{\et}(O_{K,S_{p}}, \mathbb{Z}_{p}(\boldsymbol{m}))$ is finite. 

\medskip

 Let $\mathcal{C}_{L} \coloneqq \varprojlim_{n} C(L_{n}) \otimes \mathbb{Z}_{p} \subset \mathcal{E}_{L}$. Then the Iwasawa main conjecture for imaginary quadratic fields \cite[Theorem 4.1 $(\mathrm{i})$]{Ru2} implies that the finiteness of $\mathcal{A}_{L}(\boldsymbol{m}-\boldsymbol{1})_{\mathrm{Gal}(L/K)}$ is equivalent to that of $(\mathcal{E}_{L}/\mathcal{C}_{L})(\boldsymbol{m}-\boldsymbol{1})_{\mathrm{Gal}(L/K)}$. Since $(\mathcal{E}_{L}/\mathcal{C}_{L})(\boldsymbol{m}-\boldsymbol{1})_{\mathrm{Gal}(L/K)}$ is finite and $(\mathcal{E}_{L}/\mathcal{C}_{L})(\boldsymbol{m}-\boldsymbol{1})$ is a torsion $\mathbb{Z}_{p}[[\mathrm{Gal}(L/K)]]$-module by \cite[Corollary 7.8]{Ru2}, $(\mathcal{E}_{L}/\mathcal{C}_{L})(\boldsymbol{m}-\boldsymbol{1})^{\mathrm{Gal}(L/K)}$ is also finite. By the exact sequence
\[
(\mathcal{E}_{L}/\mathcal{C}_{L})(\boldsymbol{m}-\boldsymbol{1})^{\mathrm{Gal}(L/K)} \to \mathcal{C}_{L}(\boldsymbol{m}-\boldsymbol{1})_{\mathrm{Gal}(L/K)} \to \mathcal{E}_{L}(\boldsymbol{m}-\boldsymbol{1})_{\mathrm{Gal}(L/K)} \to (\mathcal{E}_{L}/\mathcal{C}_{L})(\boldsymbol{m}-\boldsymbol{1})_{\mathrm{Gal}(L/K)} \to 0
\] (see \cite[Lemma 6.1]{Ru2}), it follows that the kernel and the cokernel of the natural homomorphism $\mathcal{C}_{L}(\boldsymbol{m}-\boldsymbol{1})_{\mathrm{Gal}(L/K)} \to \mathcal{E}_{L}(\boldsymbol{m}-\boldsymbol{1})_{\mathrm{Gal}(L/K)}$ are both finite.  

\medskip

Note that norm maps induce a homomorphism $p_{L} \colon \mathcal{C}(\boldsymbol{m}-\boldsymbol{1})_{\mathrm{Gal}(K(p^{\infty})/K)} \to \mathcal{C}_{L}(\boldsymbol{m}-\boldsymbol{1})_{\mathrm{Gal}(L/K)}$. We claim that both $\ker(p_{L})$ and $\mathrm{coker}(p_{L})$ are finite. In fact, by \cite[Theorem 7.7 $(\mathrm{ii})$]{Ru2}, it follows that the kernel and cokernel of the natural homomorphism between $\mathbb{Z}_{p}[[\mathrm{Gal}(K(p^{\infty})/K)]]$-modules
\[
\mathcal{C}_{\mathrm{Gal}(K(p^{\infty})/L)} \to \mathcal{C}_{L} \tag{$\star$}
\] are annihilated by the product of $\mathrm{ann}(\mathbb{Z}_{p}(1)) \subset \mathbb{Z}_{p}[[\mathrm{Gal}(K(p^{\infty})/K)]]$ and $\mathcal{I}_{\p}\mathcal{I}_{\bar{\p}}$. Here, $\mathcal{I}_{\p}$ (resp. $\mathcal{I}_{\bar{\p}}$) is an ideal of $\mathbb{Z}_{p}[[\mathrm{Gal}(K(p^{\infty})/K)]]$ generated by elements of the form $\sigma-1$, where $\sigma$ ranges over the decomposition group at $\p$ (resp. $\bar{\p}$). In particular, both the kernel and the cokernel of $(\star)$ are annihilated by every element of the form
\[
  (\sigma-\chi^{\boldsymbol{1}}(\sigma))(\tau_{1}-1)(\tau_{2}-1)
\] where $\tau_{1}$ (resp. $\tau_{2}$) is an element of the decomposition group of $\mathrm{Gal}(K(p^{\infty})/K)$ at $\p$ (resp. $\bar{\p}$) and $\sigma \in \mathrm{Gal}(K(p^{\infty})/K)$. 

Note that the map $p_{L}$ is obtained from $(\star)$ by taking tensor products with $\mathbb{Z}_{p}(\boldsymbol{m}-\boldsymbol{1})$ and $\mathrm{Gal}(L/K)$-coinvariants. It follows that both $\ker(p_{L})$ and $\mathrm{coker}(p_{L})$ are annihilated by every element of the following form
\[
(\chi^{\boldsymbol{1}-\boldsymbol{m}}(\sigma)-\chi^{\boldsymbol{1}}(\sigma))(\chi^{\boldsymbol{1}-\boldsymbol{m}}(\tau_{1})-1)(\chi^{\boldsymbol{1}-\boldsymbol{m}}(\tau_{2})-1) \tag{$\ast$}. 
\] Since the decomposition groups at $\p$ and $\bar{\p}$ are open in $\mathrm{Gal}(K(p^{\infty})/K)$ and $\chi^{\boldsymbol{m}}$ is nontrivial, we can choose $\sigma$, $\tau_{1}$ and $\tau_{2}$ so that ($\ast$) is nonzero. Therefore $\ker(p_{L})$ and $\mathrm{coker}(p_{L})$ are finite, as desired. 

\medskip

We have shown that both the kernel and the cokernel of the homomorphism
\[
  \mathcal{C}(\boldsymbol{m}-\boldsymbol{1})_{\mathrm{Gal}(K(p^{\infty})/K)} \to \mathcal{E}_{L}(\boldsymbol{m}-\boldsymbol{1})_{\mathrm{Gal}(L/K)}
\] are finite. Hence we have
\[
  \mathrm{rank}_{\mathbb{Z}_{p}}\mathcal{C}(\boldsymbol{m}-\boldsymbol{1})_{\mathrm{Gal}(K(p^{\infty})/K)}=\mathrm{rank}_{\mathbb{Z}_{p}}\mathcal{E}_{L}(\boldsymbol{m}-\boldsymbol{1})_{\mathrm{Gal}(L/K)} \geq 1
\] where the second inequality follows from \cite[Corollary 7.8]{Ru2}. Since the kernel of $\ker(e_{\boldsymbol{m}, L})$ is finite, it follows that the image of $\mathcal{C}(\boldsymbol{m}-\boldsymbol{1})_{\mathrm{Gal}(K(p^{\infty})/K)}$ under $e_{\boldsymbol{m}}$ is nontrivial. By Lemma \ref{non_vanishing_lemma}, we obtain the desired assertion.
\end{proof}

\subsection{On the finiteness of $H^{2}_{\et}(\mathrm{Spec}(O_{K}[\frac{1}{p}]), \mathbb{Z}_{p}(\boldsymbol{m}))$}\label{4.2}

In this subsection, we prove the finiteness of $H^{2}_{\mathrm
{\acute{e}t}}(O_{K,S_{p}}, \mathbb{Z}_{p}(\boldsymbol{m}))$ under various assumptions. First, we make some observations. Let $F$ be a subfield of $K(p)$ containing $K$, 
\[
I_{F} \coloneqq \{ \boldsymbol{m} \in I \mid \text{the Galois group $G_{F}$ acts trivially on $\mathbb{F}_{p}(\boldsymbol{m})$}
\}
\] and $\boldsymbol{m} \in I_{F}$.

Note that $H^{2}_{\et}(O_{F, S_{p}}, \mathbb{Z}_{p}(\boldsymbol{m}))$ is isomorphic to $H^{2}(G_{F, S_{p}}, \mathbb{Z}_{p}(\boldsymbol{m}))$, where the latter is a Galois cohomology group of $G_{F, S_{p}} \coloneqq \pi_{1}^{\et}(\mathrm{Spec}(O_{F, S_{p}}))$, the maximal Galois group of $F$ unramified outside $S_{p}$ by \cite[$\mathrm{II}$, Proposition 2.9]{Mi}. Moreover, by \cite[Corollary 10.4.8]{NSW}, there is a natural isomorphism 
\[
H^{2}(G_{F, S_{p}}^{(p)}, \mathbb{Z}_{p}(\boldsymbol{m})) \xrightarrow{\sim} H^{2}(G_{F, S_{p}}, \mathbb{Z}_{p}(\boldsymbol{m}))
\] since $\boldsymbol{m} \in I_{F}$.

\begin{lemma}\label{k_vs_F}
Let $F$ be a subfield of $K(p)$ containing $K$ and $\boldsymbol{m} \in I$. There is an isomorphism
\[
H^{2}_{\et}(O_{K, S_{p}}, \mathbb{Z}_{p}(\boldsymbol{m})) \xrightarrow{\sim} H^{2}_{\et}(O_{F, S_{p}}, \mathbb{Z}_{p}(\boldsymbol{m}))^{\mathrm{Gal}(F/K)}
\]
\end{lemma} 
\begin{proof}
There is the Hochschild-Serre spectral sequence 
\[
E^{p,q}_{2}=H^{p}(\mathrm{Gal}(F/K), H^{q}_{\et}(O_{F, S_{p}}, \mathbb{Z}_{p}/p^{n}(\boldsymbol{m}))) \Rightarrow H^{p+q}_{\et}(O_{K,S_{p}}, \mathbb{Z}_{p}/p^{n}(\boldsymbol{m}))
\] for each $n \geq 1$. Since $\mathrm{Gal}(F/K)$ is a finite prime-to-$p$ group, $E^{1,1}_{2}$, $E^{1,2}_{2}$ and $E^{2,0}_{2}$ are trivial. Hence we obtain 
\[
  H^{2}_{\et}(O_{K, S_{p}}, \mathbb{Z}_{p}/p^{n}(\boldsymbol{m})) \xrightarrow{\sim} H^{2}_{\et}(O_{F, S_{p}}, \mathbb{Z}_{p}/p^{n}(\boldsymbol{m}))^{\mathrm{Gal}(F/K)}
\] and taking the limits yields the desired assertion.
\end{proof}

By Lemma \ref{k_vs_F}, we are reduced to considering $H^{2}(G_{F, S_{p}}^{(p)}, \mathbb{Z}_{p}(\boldsymbol{m}))$.  In the rest of this section, we derive the desired finiteness in the following three situations : (1) $\boldsymbol{m} \in (p-1)\mathbb{Z}^{2}$, (2) $\boldsymbol{m}=(m,m)$ for $m \geq 2$ and (3) every $\boldsymbol{m} \in I_{F}$ when $p$ is a ``purely local'' prime. 

\medskip

We consider the first case in the following lemma. 

\begin{lemma}\label{free}
The Galois group $G_{K, S_{p}}^{(p)}$ is a free pro-$p$ group of rank two. In particular, $H^{2}(G_{K, S_{p}}^{(p)}, \mathbb{Z}_{p}(\boldsymbol{m}))$ is trivial for every $\boldsymbol{m} \in I_{K}$.
\end{lemma}
\begin{proof}
By \cite[Theorem 10.7.13]{NSW}, it holds that $\dim H^{1}(G_{K, S_{p}}^{(p)}, \mathbb{F}_{p})=2+\dim V_{S_{p}}(K)$ and $\dim H^{2}(G_{K, S_{p}}^{(p)}, \mathbb{F}_{p})=\dim V_{S_{p}}(K)
$, where
\[
V_{S_{p}}(K) \coloneqq \ker \left( K^{\times}/K^{\times p} 
\to \prod_{v \in S_{p}} K_{v}^{\times}/K_{v}^{\times p} \times  \prod_{v \not \in S_{p}} K_{v}^{\times}/U_{v} K_{v}^{\times p} 
\right)
\]and $U_{v}$ is the group of principal units in $O_{K_{v}}$. Note that $U_{v}$ is contained in  $(K_{v}^{\times})^{p}$ for every $v \not \in S_{p}$. Since every element of $K^{\times}$ which is a $p$-th power in $K_{v}^{\times}$ for every finite place $v$ is globally a $p$-th power up to $O_{K}^{\times}/O_{K}^{\times p}$, the group $V_{S_p}(K)$ is trivial as desired. 
\end{proof}

This lemma and Theorem \ref{non_vanishing} imply:

\begin{corollary}\label{cor:p-1}
The elliptic Soul\'e character $\kappa_{\boldsymbol{m}}$ is nontrivial for every $\boldsymbol{m} \in I_{K}=(p-1)\mathbb{Z}_{\geq 1}^{2}$.
\end{corollary}

As for the next case, we have the following unconditional result due to Soul\'e.

\begin{theorem}[Soul\'e {\cite[page 287, Corollaire]{So3}}]\label{Soule}
\[
H^{2}_{\et}(O_{K(\mu_{p}), S_{p}}, \mathbb{Z}_{p}(m))
\] is finite for every $m \geq 2$.
\end{theorem}

By this theorem and Theorem \ref{non_vanishing}, we have:

\begin{corollary}\label{Soule_cyc}
The elliptic Soul\'e character $\kappa_{(m,m)}$ is nontrivial for every $m \geq 2$.
\end{corollary}

\begin{remark}
If $m \geq 3$ is odd, it is natural to consider whether the elliptic Soul\'e character $\kappa_{(m,m)}$ can be written in terms of the $m$-th Soul\'e character $\kappa_{m}$. In a subsequent paper, we show that $\kappa_{(m,m)}$ is indeed equal to $\kappa_{m}$ times an explicit constant, by using a result of Kersey \cite{Ka}. 
\end{remark}

For the last case, we prove the finiteness of $H^{2}_{\et}(O_{F, S_{p}}, \mathbb{Z}_{p}(\boldsymbol{m}))$ for every $\boldsymbol{m} \in I_{F}$ when $p$ satisfies a certain condition which is analogous to the regularity of prime numbers.

\begin{definition}\label{purely_local}
Let $F$ be a subfield of $K(E[p])$ containing $K$, $v$ a prime of $F$ above $p$ and $F_{v}(p)$ the maximal pro-$p$ extension of $F_{v}$. We say $v$ is \emph{purely local} if the natural map
\[ \mathrm{Gal}(F_{v}(p)/F_{v}) \to \mathrm{Gal}(F_{S_{p}}(p)/F)
\] is an isomorphism.
\end{definition}

The pure locality is a strong condition as the following lemma suggests.

\begin{lemma}\label{regular_cl}
Let $F$ be a subfield of $K(E[p])$ containing $K$ and assume that $v  \in S_{\p}(F)$ is purely local. Then the following assertions hold.
\begin{enumerate}
\item There exists a unique prime of $F$ above $\p$, i.e. $S_{\p}(F)=\{ v \}$. Moreover, if $F$ contains $\mu_{p}$, there exists a unique prime of $F$ above $\bar{\p}$.
\item If $F$ contains $\mu_{p}$, the class number of $F$ is not divisible by $p$.
\end{enumerate}
\end{lemma}
\begin{proof}
(1) We compare the dimensions of $H^{1}$ with coefficients in $\mathbb{F}_{p}$ of $\mathrm{Gal}(F_{v}(p)/F_{v})$ and $\mathrm{Gal}(F_{S_{p}}(p)/F)$. First, by local class field theory, we have
\[
\dim_{\mathbb{F}_{p}} H^{1}(\mathrm{Gal}(F_{v}(p)/F_{v}), \mathbb{F}_{p})=1+\delta_{v}+[F_{v}: \mathbb{Q}_{p}].
\] Here, $\delta_{v}$ is  $1$ if $F_{v}$ contains $\mu_{p}$ and otherwise $0$. For the latter cohomology group, by \cite[Theorem 10.7.13]{NSW}, it holds that 
\[
\dim_{\mathbb{F}_{p}} H^{1}(\mathrm{Gal}(F_{S_{p}}(p)/F), \mathbb{F}_{p}) \geq 1-\delta+[F:K]+\sum_{w \in S_{p}(F)} \delta_{w}.
\] Here, $\delta$ is  $1$ if $F$ contains $\mu_{p}$ and otherwise $0$. The definition of $\delta_{w}$ is similar to that of $\delta_{v}$. Since $v$ is purely local, these two (in)equalities imply that
\[
\delta_{v}+[F_{v}: \mathbb{Q}_{p}] \geq -\delta+[F:K]+\sum_{w \in S_{p}(F)} \delta_{w},
\] which is equivalent to the inequality
\[
\delta \geq \left( [F : K]-[F_{v}: \mathbb{Q}_{p}] \right)+\sum_{w \in S_{p}(F) \setminus \{ v \} } \delta_{w}.
\] If $F$ does not contain $\mu_{p}$, this inequality implies that $[F_{v}: \mathbb{Q}_{p}]=[F:K]$, which is equivalent to saying that $S_{\p}(F)=\{ v \}$. If $F$ contains  $\mu_{p}$, it follows by this inequality that $[F_{v}: \mathbb{Q}_{p}]=[F:K]$ and the set $S_{p}(F) \setminus \{ v \}$ is a singleton. Hence the assertion follows.

(2) Since $F$ contains $\mu_{p}$,  by \cite[Theorem 10.7.13]{NSW} and (1), it holds that 
\[
\dim_{\mathbb{F}_{p}} H^{1}(\mathrm{Gal}(F_{S_{p}}(p)/F), \mathbb{F}_{p})=2+[F:K]+\dim Cl_{S_{p}}(F)/p
\] where $Cl_{S_{p}}(F)$ is a quotient of the usual ideal class group $Cl(F)$ divided by a subgroup generated by $S_{p}(F)$. By comparing again, we have $Cl_{S_{p}}(F)/p=0$. Since primes of $F$ above $\p$ and $\bar{\p}$ are unique by (1) and $[F:K]$ is prime to $p$, the orders of such two primes in $Cl(F)$ are not divisible by $p$. Therefore the class number of $F$ is not divisible by $p$ as desired.
\end{proof}

For a prime number $p$, one can observe that $p$ is regular if and only if there exists a unique prime of $F_{S_{p}}(p)$ above $p$ where $F=\mathbb{Q}(\mu_{p})$. This fact, together with the following proposition, explains a certain similarity between regular primes and purely local primes. 

\begin{proposition}
Let $F$ be a subfield of $K(E[p])$ containing $K$ and $v$ a prime of $F$ above $\p$. The following assertions are equivalent.
\begin{enumerate}
\item $v$ is purely local.
\item The natural map $\mathrm{Gal}(F_{v}(p)/F_{v}) \to \mathrm{Gal}(F_{S_{p}}(p)/F)$ is surjective.
\item There exists a unique prime of $F_{S_{p}}(p)$ above $\p$.
\end{enumerate}
\end{proposition}
\begin{proof}
$(1) \Rightarrow (2)$ and $(3) \Rightarrow (2)$: clear.  $(2) \Rightarrow (3)$: The same argument as the proof of Lemma \ref{regular_cl} (1) shows that $(2)$ implies that there exists a unique prime of $F$ above $\p$. Hence a prime of $F_{S_{p}}(p)$ above $\p$ is also unique. 

$(2) \Rightarrow (1)$ : In the proof of \cite[Theorem]{Wi}, Wingberg already proved this implication though not stated explicitly. In the following, we simply reproduce his proof for readers' convenience. Let $H$ be the kernel of the concerned homomorphism. By the Hochschild-Serre spectral sequence, 
\[
0 \to H^{1}(\mathrm{Gal}(F_{S_{p}}(p)/F), \mathbb{Q}_{p}/\mathbb{Z}_{p}) \to H^{1}(\mathrm{Gal}(F_{v}(p)/F_{v}), \mathbb{Q}_{p}/\mathbb{Z}_{p}) \to  H^{1}(H, \mathbb{Q}_{p}/\mathbb{Z}_{p})^{\mathrm{Gal}(F_{S_{p}}(p)/F)} \to 0
\] is exact since $H^{2}(\mathrm{Gal}(F_{S_{p}}(p)/F), \mathbb{Q}_{p}/\mathbb{Z}_{p})=0$ by Leopoldt's conjecture for $F$. 

To prove the triviality of $H$, it suffices to show that $H^{1}(H, \mathbb{Q}_{p}/\mathbb{Z}_{p})^{\mathrm{Gal}(F_{S_{p}}(p)/F)}$ is trivial. In fact, since 
\[
H^{1}(H, \mathbb{Q}_{p}/\mathbb{Z}_{p})^{\mathrm{Gal}(F_{S_{p}}(p)/F)}=
\mathrm{Hom}(H/[H, \mathrm{Gal}(F_{v}(p)/F_{v})], \mathbb{Q}_{p}/\mathbb{Z}_{p}),
\] we then have $H=[H, \mathrm{Gal}(F_{v}(p)/F_{v})]$. In particular, $H$ is contained in the intersection of the descending central series of $\mathrm{Gal}(F_{v}(p)/F_{v})$ which is trivial. Hence we are reduced to showing that the homomorphism
\[
  H^{1}(\mathrm{Gal}(F_{S_{p}}(p)/F), \mathbb{Q}_{p}/\mathbb{Z}_{p}) \to H^{1}(\mathrm{Gal}(F_{v}(p)/F_{v}), \mathbb{Q}_{p}/\mathbb{Z}_{p})
\] is surjective. By Pontryagin duality, this is equivalent to saying that the natural map
\[
  \mathrm{Gal}(F_{v}(p)/F_{v})^{\mathrm{ab}}
  \to
  \mathrm{Gal}(F_{S_{p}}(p)/F)^{\mathrm{ab}}
  \tag{$\ast$}
\] is injective. First note that 
\[
\mathrm{rank}_{\mathbb{Z}_{p}} \mathrm{Gal}(F_{v}(p)/F_{v})^{\mathrm{ab}}
=
\mathrm{rank}_{\mathbb{Z}_{p}} \mathrm{Gal}(F_{S_{p}}(p)/F)^{\mathrm{ab}}
=
[F:K]+1
\] by local class field theory, Leopoldt's conjecture for $F$ and $(2) \Rightarrow (3)$. Hence the map $(\ast)$ induces an isomorphism on the free part. 

Note that the torsion subgroup of $\mathrm{Gal}(F_{v}(p)/F_{v})^{\mathrm{ab}}$ is cyclic of order $p^{\delta_{v}}$. Hence the torsion subgroup of $\mathrm{Gal}(F_{S_{p}}(p)/F)^{\mathrm{ab}}$ is cyclic of order dividing $p^{\delta_{v}}$ since $(\ast)$ is surjective. Moreover, we have
\[
\dim_{\mathbb{F}_{p}} H^{1}(\mathrm{Gal}(F_{S_{p}}(p)/F), \mathbb{F}_{p})
\geq \dim_{\mathbb{F}_{p}} H^{1}(\mathrm{Gal}(F_{v}(p)/F_{v}), \mathbb{F}_{p})=[F : K] + 1 + \delta_{v}
\] by local class field theory and \cite[Theorem 10.7.13]{NSW}. Hence the orders of the torsion subgroups of both sides of $(\ast)$ must coincide. 
\end{proof}

\begin{proposition}\label{prp:purely_local}
Let $F$ be a subfield of $K(p)$ containing $K$ and $\boldsymbol{m} \in I_{F}$. If one of $\p$ or $\bar{\p}$ is purely local for $F$, then $H^{2}_{\et}(O_{F, S_{p}}, \mathbb{Z}_{p}(\boldsymbol{m}))$ is finite. In particular, the $\boldsymbol{m}$-th elliptic Soul\'e character $\kappa_{\boldsymbol{m}}$ is nontrivial.
\end{proposition}
\begin{proof}
First, recall that $H^{2}_{\et}(O_{F, S_{p}}, \mathbb{Z}_{p}(\boldsymbol{m}))$ is isomorphic to $H^{2}(G_{F, S_{p}}^{(p)}, \mathbb{Z}_{p}(\boldsymbol{m}))$. If one of $\p$ or $\bar{\p}$ is purely local for $F$, $G_{F, S_{p}}^{(p)}$ is a free pro-$p$ group if $\mu_{p} \not \subset F$ or a pro-$p$ Demushkin group (cf. \cite[Definition 3.9.9]{NSW} for the definition) otherwise by \cite[Theorem 7.5.11]{NSW}.  For the former case, $H^{2}(G_{F, S_{p}}^{(p)}, \mathbb{Z}_{p}(\boldsymbol{m}))$ is trivial. For the latter, the cohomology group is also finite since we have
\[
H^{2}(G_{F, S_{p}}^{(p)}, \mathbb{Z}_{p}(\boldsymbol{m}))
= \mathbb{Z}_{p}(\boldsymbol{m}-\boldsymbol{1})_{\mathrm{Gal}(F_{\infty}/F)}
\] where $F_{\infty}$ is the unique $\mathbb{Z}_{p}^{2}$-extension of $F$ inside $K(p^{\infty})$ by the Tate duality. The rest of the assertion follows from Theorem \ref{non_vanishing}.
\end{proof}

In the following, we briefly discuss a work of Wingberg which relates pure locality of primes to a work of Yager \cite{Ya2} on ``regular'' primes for (extensions of) imaginary quadratic fields.

\begin{definition}\label{regular_F}
Let $F$ be a subfield of $K(E[p])$ containing $K$. We say $\p$ (resp. $\bar{\p}$) is \emph{regular for $F$} if the field $F_{S_{\p}}(p)$ (resp. $F_{S_{\bar{\p}}}(p)$) is a $\mathbb{Z}_{p}$-extension of $F$. We say $p$ is regular for $F$ if $\p$ and $\bar{\p}$ are both regular for $F$.
\end{definition}

One can easily observe the following lemma.

\begin{lemma}\label{regular_subfield}
The following assertions hold.
\begin{enumerate}
\item $p$ is regular for $K$.
\item Let $K \subset F \subset L \subset K(E[p])$ be subfields of $K(E[p])$ containing $K$. If $\p$ is regular for $L$, then $\p$ is so for $F$.
\end{enumerate}
\end{lemma}
\begin{proof}
(1) Note that the maximal abelian quotient of $\mathrm{Gal}(K_{S_{\p}}(p)/K)$ is equal to the maximal pro-$p$ quotient of $\mathrm{Gal}(K(\p^{\infty})/K)$, which is isomorphic to $\mathbb{Z}_{p}$. Hence $\mathrm{Gal}(K_{S_{\p}}(p)/K)$ itself is isomorphic to $\mathbb{Z}_{p}$. The same argument shows that $\bar{\p}$ is also regular for $K$.

(2) Note that $F_{S_{\p}}(p)$ contains a $\mathbb{Z}_{p}$-extension $K_{S_{\p}}(p)\cdot F$ of $F$. Since the natural homomorphism $\mathbb{Z}_{p} \cong \mathrm{Gal}(L_{S_{\p}}(p)/L) \to \mathrm{Gal}(F_{S_{\p}}(p)/F)$ is surjective, the assertion follows.
\end{proof}

We introduce a theorem of Wingberg. As we shall see, the following original statement needs a slight modification since there is a counter-example (Example \ref{counter_example}).

\begin{theorem}[Wingberg {\cite[Theorem]{Wi}}]\label{Wingberg}
Let $F$ be a subfield of $K(E[p])$ containing $K$. Then $\p$ is regular for $F$ if and only if a prime of $F$ above $\bar{\p}$ is purely local. A similar statement also holds when replacing $\p$ by $\bar{\p}$.
\end{theorem}

If this theorem is true, then $\p$ would be always purely local for $K$. However, the following example gives a counter-example. We remark that both $G_{K_{\p}}^{(p)}$ and $\mathrm{Gal}(K_{S_{p}}(p)/K)$ are free pro-$p$ groups of rank two, hence are isomorphic as abstract profinite groups.

\begin{example}\label{counter_example}
Let $F=K=\mathbb{Q}(\sqrt{-1})$ and $p=29789$. If we set $\pi=110+133\sqrt{-1}$ and $\p=(\pi)$, then we have $(p)=\p\bar{\p}$. In the following, we show that the natural map $G_{K_{\p}}^{(p)} \to \mathrm{Gal}(K_{S_{p}}(p)/K)$ is not surjective by using SageMath \cite{Sage}. This is the only counter-example we found among primes less than $50000$.

To confirm the claim, it suffices to show that the map
\[
G_{K_{\p}}^{(p)} \to \mathrm{Gal}(K_{S_{\bar{\p}}}(p)/K)
\] is not surjective. In other words, the Frobenius element $\sigma_{\p}$ does not generate $\mathrm{Gal}(K_{S_{\bar{\p}}}(p)/K)$. Note that $\mathrm{Gal}(K_{S_{\bar{\p}}}(p)/K)$ is nothing but the group of principal units in $\mathrm{Gal}(K(\bar{\p}^{\infty})/K) \cong O_{K_{\bar{\p}}}^{\times}/O_{K}^{\times}$ and $\sigma_{\p}$ maps to $\pi \bmod O_{K}^{\times}$ under this identification. 

Now the projection of $\pi$ onto the group of principal unit is a generator if and only if its $(p-1)$-th power is so. We calculated the value $\pi^{p-1}-1 \bmod \bar{\p}^{2}$ by using SageMath and found that $\pi^{p-1}$ is contained in the subgroup $1+\bar{\p}^{2}O_{K_{\bar{\p}}} \subsetneq 1+\bar{\p}O_{K_{\bar{\p}}}$, which confirms the claim. 
\end{example}

In fact, the proof of \cite[Theorem]{Wi} essentially shows the following slightly weaker statement than the original one:

\begin{theorem}[A modified version of Theorem {\ref{Wingberg}}]\label{thm:modify}
Let $F$ be a subfield of $K(E[p])$ containing $K$. Then $\p$ is regular for $F$ and there exists a unique prime of $F_{S_{\p}}(p)$ above $\bar{\p}$ if and only if a prime of $F$ above $\bar{\p}$ is purely local. A similar statement also holds when replacing $\p$ by $\bar{\p}$.
\end{theorem}

\begin{remark}
Assume that $F$ is a subfield of $K(E[p])$ containing $K$. If there exists a unique prime of $F_{S_{\p}}(p)$ above $\bar{\p}$, then $K_{S_{\p}}(p)$ also has a unique prime above $\bar{\p}$. The latter is equivalent to saying that $\mathrm{Frob}_{\bar{\p}}$ generates the group of principal units in $\mathrm{Gal}(K(\p^{\infty})/K) \cong O_{K_{\p}}^{\times}/O_{K}^{\times}$, which does not hold in general as Example \ref{counter_example} shows.
\end{remark}

\begin{example}
In \cite[page 33]{Ya2}, Yager provided a Kummer-type criterion which relates the regularity of $\p$ to $\p$-adic properties of special values of Hecke $L$-functions and gave examples of regular primes when $K=\mathbb{Q}(\sqrt{-1})$. For example, $p=5$ is regular for the mod-$5$ ray class field $K(5)$. 

As an another example, $p=29$ is regular for the mod-$\p$ ray class field $K(\p)$ although it is not for $K(29)$. In this case, the prime over $\p$ is purely local for $K(\p)$. Hence we can apply Proposition \ref{prp:purely_local} to conclude the nontriviality of the corresponding elliptic Soul\'e characters.
\end{example}

\section{Surjectivity of elliptic Soul\'e characters}\label{5}

In this last section, we consider the elliptic Soul\'e characters modulo $p$ and derive nontriviality of them under certain situations.

First, by the definition of the elliptic Soul\'e characters (cf. Definition \ref{ellSoule}), $\kappa_{\boldsymbol{m}}$ is congruent to $\kappa_{\boldsymbol{n}}$ modulo $p$ for every $\boldsymbol{m}, \boldsymbol{n} \geq (2,2)$ if $\boldsymbol{m} \equiv \boldsymbol{n} \bmod p-1$. Similarly, $\kappa_{(m,1)}$ is congruent to $\kappa_{(n,1)}$ modulo $p$ and $\kappa_{(1,m)}$ is congruent to $\kappa_{(1,n)}$ modulo $p$ for every $m, n \geq 2$ if $m$ is congruent to $n$ modulo $p-1$. 

\medskip

Let us fix $\boldsymbol{m}=(m_{1}, m_{2}) \in I$. The reduction modulo $p$ of $\kappa_{\boldsymbol{m}}$ is nontrivial if and only if the element 
\[
  \epsilon_{\boldsymbol{m}, 1}=\prod_{\substack{0 \leq a,b <p \\ p \nmid \gcd(a,b)}}
\left(
\prod_{k=0}^{n} \theta_{p}(a\omega_{1, 2}+b\omega_{2, 2})\right)^{a^{m_{1}-1}b^{m_{2}-1}}
\] is nontrivial as an element of $K(p^{\infty})^{\times}/K(p^{\infty})^{\times p}$. By Lemma \ref{mathfrak_a}, we have the following equality
\[
\epsilon_{\boldsymbol{m}, 1}^{N\mathfrak{a}-\chi^{\boldsymbol{1}-\boldsymbol{m}}(\sigma_{\mathfrak{a}})}=\prod_{\substack{0 \leq a,b <p \\ p \nmid \gcd(a,b)}}\theta_{\mathfrak{a}}(a\omega_{1,1}+b\omega_{2,1})^{a^{m_{1}-1}b^{m_{2}-1}}
\] in  $K(p^{\infty})^{\times}/K(p^{\infty})^{\times p}$ for every ideal $\mathfrak{a}$ prime to $p$. 

By Lemma \ref{non_deg_case} and Lemma \ref{deg_case}, the element $\epsilon_{\boldsymbol{m}, 1}$ times $N\mathfrak{a}-\chi^{\boldsymbol{1}-\boldsymbol{m}}(\sigma_{\mathfrak{a}})$ is nontrivial if and only if $\epsilon_{\boldsymbol{m}, 1, \mathfrak{a}}$ is so as an element of $K(p^{\infty})^{\times}/K(p^{\infty})^{\times p}$ : recall that, in these lammas, $\epsilon_{\boldsymbol{m}, 1, \mathfrak{a}}$ was defined as 
\[
  \epsilon_{\boldsymbol{m}, 1, \mathfrak{a}} \coloneqq \prod_{a, b \in (\mathbb{Z}/p \mathbb{Z})^{\times}}\theta_{\mathfrak{a}}(a\omega_{1,1}+b\omega_{2,1})^{a^{m_{1}-1}b^{m_{2}-1}}
\] if $\boldsymbol{m} \geq (2,2)$, and
\[
  \epsilon_{\boldsymbol{m},1, \mathfrak{a}} \coloneqq \prod_{a \in (\mathbb{Z}/p\mathbb{Z})^{\times}} \theta_{\mathfrak{a}}(a\bar{\pi}\omega_{1,1})^{a^{m_{1}-1}}
\] if $m_{2}=1$ and the remaining case is similar. In fact, the term
\[
  \frac{1-i^{\boldsymbol{m}-\boldsymbol{1}}(\bar{\pi})}{1-i^{\boldsymbol{m}-\boldsymbol{1}}(\pi)}+i^{\boldsymbol{m}-\boldsymbol{1}}(\bar{\pi})
\] appearing in Lemma \ref{non_deg_case} is contained in $1+p\mathbb{Z}_{p}$ if $\boldsymbol{m} \geq (2,2)$. Moreover, the term
\[
  \frac{1-i^{\boldsymbol{m}-\boldsymbol{1}}(\bar{\pi})}{1-i^{\boldsymbol{m}-\boldsymbol{1}}(\pi)}+i^{\boldsymbol{m}-\boldsymbol{1}}(\bar{\pi})=\frac{1-i^{\boldsymbol{m}-\boldsymbol{1}}(p)}{1-i^{\boldsymbol{m}-\boldsymbol{1}}(\pi)}
\] appearing in Lemma \ref{deg_case} is also contained in $1+p\mathbb{Z}_{p}$ if $m_{2}=1$ (the case where $m_{1}=1$ is similar). 

\medskip

In the following, we set $\omega \coloneqq \omega_{1,1}+\omega_{2,1}$ and
\[
 \phi_{\boldsymbol{m}} \coloneqq \sum_{\sigma \in \mathrm{Gal}(K(p)/K)} \chi^{\boldsymbol{m}-\boldsymbol{1}}(\sigma) \sigma \in \mathbb{F}_{p}[\mathrm{Gal}(K(p)/K)],
\] which projects $K(p)^{\times}/K(p)^{\times p}$ onto its $\chi^{\boldsymbol{1}-\boldsymbol{m}}$-th isotypic component. It follows from Proposition \ref{elliptic_unit} (2) that 
\[
  \epsilon_{\boldsymbol{m}, 1, \mathfrak{a}}=w_{K}\phi_{\boldsymbol{m}} \left( \theta_{\mathfrak{a}}(\omega) \right)
\] if $\boldsymbol{m} \geq (2,2)$ and that
\[
  \epsilon_{\boldsymbol{m}, 1, \mathfrak{a}}=w_{K}\phi_{\boldsymbol{m}} \left( \theta_{\mathfrak{a}}(\bar{\pi}\omega) \right)
\] if $m_{2}=1$.

\begin{lemma}\label{infty_to_one}
  Assume that $\boldsymbol{m} \in I$ is not congruent to $(0,0)$ modulo $p-1$ and  $\mathfrak{a}$ satisfies $N\mathfrak{a}-\chi^{\boldsymbol{1}-\boldsymbol{m}}(\sigma_{\mathfrak{a}}) \not \equiv 0 \bmod p$. Then $\kappa_{\boldsymbol{m}}$ is surjective if and only if the element \[
      \epsilon_{\boldsymbol{m}, 1, \mathfrak{a}} \in K(p)^{\times}/K(p)^{\times p}
    \] is nontrivial.

\end{lemma}

\begin{proof}
Note that the kernel of the natural homomorphism 
\[
  K(p)^{\times}/K(p)^{\times p} \to K(p^{\infty})^{\times}/K(p^{\infty})^{\times p} 
\] is isomorphic to $H^{0}(\mathrm{Gal}(K(p^{\infty})/K(p)), \mu_{p})=\mu_{p}$. Since $\epsilon_{\boldsymbol{m}, 1, \mathfrak{a}}$ is contained in the $\chi^{\boldsymbol{1}-\boldsymbol{m}}(\neq \chi^{\boldsymbol{1}})$-isotypic component, it follows that
\[
  \epsilon_{\boldsymbol{m}, 1, \mathfrak{a}} \in K(p^{\infty})^{\times}/K(p^{\infty})^{\times p}
\] is nontrivial if and only if it is so as an element of $K(p)^{\times}/K(p)^{\times p}$. This observation, together with the assumption $N\mathfrak{a}-\chi^{\boldsymbol{1}-\boldsymbol{m}}(\sigma_{\mathfrak{a}}) \not \equiv 0 \bmod p$, concludes the proof.
\end{proof}

In the following, we treat $\kappa_{\boldsymbol{m}}$ separately according to the index $\boldsymbol{m} \in I$:

\medskip

Case 1:  $\boldsymbol{m} \geq (2,2)$ and $\boldsymbol{m} \equiv \boldsymbol{1} \bmod p-1$.

\medskip

Case 2:  $\boldsymbol{m} \geq (2,2)$ and $\boldsymbol{m} \not \equiv \boldsymbol{1} \bmod p-1$.

\medskip

Case 3:  either $m_{1}=1$ or $m_{2}=1$ and $\boldsymbol{m} \equiv \boldsymbol{1} \bmod p-1$.

\medskip

Case 4: either $m_{1}=1$ or $m_{2}=1$ and $\boldsymbol{m} \not \equiv \boldsymbol{1} \bmod p-1$.

\medskip

Recall that the group of units $O_{K(p)}^{\times}$ (resp. $O_{K(\p)}^{\times}$, $O_{K(\bar{\p})}^{\times}$) is abbreviated as $E_{1}$ (resp. $E_{(1,0)}$, $E_{(0,1)}$) cf. the beginning of Section \ref{4}.

\subsection{First case}\label{5.1}

We assume that $\boldsymbol{m} \geq (2,2)$ and $\boldsymbol{m} \equiv (1,1) \bmod p-1$. It suffices to consider the case where $\boldsymbol{m}=(p,p)$.

Choose an ideal $\mathfrak{a}$ prime to $p$ such that the first condition in Lemma \ref{infty_to_one} holds. Then the element $\epsilon_{\boldsymbol{m},1, \mathfrak{a}}$ is contained in the group
\[
  \left( E_{1}/E_{1}^{p} \right)^{\mathrm{Gal}(K(p)/K)}=O_{K}^{\times}/O_{K}^{\times p}
\] which is trivial since $p \geq 5$. By Lemma \ref{infty_to_one}, we have:

\begin{proposition}\label{surj_11}
Let $\boldsymbol{m}=(m_{1}, m_{2}) \in I$ such that  $\boldsymbol{m} \geq (2,2)$ and $\boldsymbol{m} \equiv \boldsymbol{1} \mod p-1$. The $\boldsymbol{m}$-th elliptic Soul\'e character $\kappa_{\boldsymbol{m}}$ is not surjective.
\end{proposition}

\subsection{Second case}\label{5.2}

We assume that $\boldsymbol{m}=(m_{1}, m_{2}) \in I$ satisfies $\boldsymbol{m} \geq (2,2)$ and $\boldsymbol{m} \not \equiv \boldsymbol{1} \bmod p-1$.

\begin{proposition}\label{surj_center}
Assume that $m_{1}, m_{2} \not \equiv 1 \bmod p-1$ and that $\boldsymbol{m} \not \equiv (0,0) \bmod p-1$. If the class number of $K(p)$ is not divisible by $p$, then the $\boldsymbol{m}$-th elliptic Soul\'e character $\kappa_{\boldsymbol{m}}$ is surjective (in particular, nontrivial). 
\end{proposition}

\begin{proof}
Since the class number of $K(p)$ is not divisible by $p$, the equality 
\[
C(K(p))/\left( C(K(p)) \cap E_{1}^{p} \right)= E_{1}/E_{1}^{p} \tag{$\ast$}
\] holds by Theorem \ref{Rubin_cl}. By Dirichlet's unit theorem \cite[Proposition 8.7.2]{NSW}, the $\chi^{\boldsymbol{1}-\boldsymbol{m}}$-component of $E_{1}/E_{1}^{p}$ is observed to be one-dimensional. 

Recall that $C(K(p))$ is generated by the image of the three subgroups
\[
  C'(K(p)), \quad D_{1}(K(p)) \quad \text{and} \quad D_{2}(K(p)),
\] see Definition \ref{C_D}. Since $D_{1}(K(p))$ is $\mathrm{Gal}(K(p)/K(\p))$-invariant, it follows that the $\chi^{\boldsymbol{1}-\boldsymbol{m}}$-isotypic component of the image of $D_{1}(K(p))$ is trivial. A similar argument can be applied to $D_{2}(K(p))$.

By definition, the group $C'(K(p))$ is generated by $\mu(K(p))$ and $\theta_{\mathfrak{a}}(\omega)^{\sigma-1}$ where $\mathfrak{a}$ ranges over the ideals relatively prime to $6p$ and $\sigma$ over $\mathrm{Gal}(K(p)/K)$. If we consider the action of 
\[
 \phi_{\boldsymbol{m}}=\sum_{\sigma \in \mathrm{Gal}(K(p)/K)} \chi^{\boldsymbol{m}-\boldsymbol{1}}(\sigma) \sigma \in \mathbb{F}_{p}[\mathrm{Gal}(K(p)/K)]
\] on $E_{1}/E_{1}^{p}$ (cf. the discussion before Lemma \ref{infty_to_one}), it follows from the equality ($\ast$) that the $\chi^{\boldsymbol{1}-\boldsymbol{m}}$-isotypic component of $E_{1}/E_{1}^{p}$ is generated by  $\phi_{\boldsymbol{m}}(\theta_{\mathfrak{a}}(\omega))$, where $\mathfrak{a}$ ranges over the  ideals relatively prime to $6p$. We note that $\mu(K(p))$ does not affect the $\chi^{\boldsymbol{1}-\boldsymbol{m}}$-isotypic component since we assume $\boldsymbol{m} \not \equiv (0,0) \bmod p-1$.

Fix an ideal $\mathfrak{a}$ prime to $6p$ such that $\chi^{\boldsymbol{1}-\boldsymbol{m}}(\sigma_{\mathfrak{a}})-N\mathfrak{a} \not \equiv 0 \bmod p$. Then, for an arbitrary ideal $\mathfrak{c}$ prime to $6p$, we have
\[
\theta_{\mathfrak{a}}(\omega)^{\sigma_{\mathfrak{c}}-N\mathfrak{c}}
=
\theta_{\mathfrak{c}}(\omega)^{\sigma_{\mathfrak{a}}-N\mathfrak{a}}
\] by Lemma \ref{p_to_a}. Hence it follows that
\[
(\chi^{\boldsymbol{1}-\boldsymbol{m}}(\sigma_{\mathfrak{c}})-N\mathfrak{c}) \phi_{\boldsymbol{m}}(\theta_{\mathfrak{a}}(\omega))=(\chi^{\boldsymbol{1}-\boldsymbol{m}}(\sigma_{\mathfrak{a}})-N\mathfrak{a})\phi_{\boldsymbol{m}}(\theta_{\mathfrak{c}}(\omega)).
\] This shows that the $\chi^{\boldsymbol{1}-\boldsymbol{m}}$-component of $E_{1}/E_{1}^{p}$ is generated by a single element $\phi_{\boldsymbol{m}}(\theta_{\mathfrak{a}}(\omega))$, which is nontrivial by the equality ($\ast$). Hence we can apply Lemma \ref{infty_to_one} to conclude the proof.
\end{proof}

If $m_{1}$ is congruent to $1 \bmod p-1$, the above argument does not work since the $\chi^{\boldsymbol{1}-\boldsymbol{m}}$-isotypic component of the image of $D_{2}(K(p))$ in $E_{1}/E_{1}^{p}$ can be nontrivial. In this case, we prove the following

\begin{proposition}\label{surj_center_2}
   Assume that $m_{1} \equiv 1 \bmod p-1$ or $m_{2} \equiv 1 \bmod p-1$. If the class number of $K(p)$ is not divisible by $p$ and there exists a unique prime of $K(p)$ above $\p$, then $\kappa_{\boldsymbol{m}}$ is surjective (in particular, nontrivial).
\end{proposition}

\begin{proof}
We may assume that $m_{1} \equiv 1 \bmod p-1$ and $m_{2} \not \equiv 1 \bmod p-1$.  Since the class number of $K(p)$ is not divisible by $p$, the group $E_{1}/E_{1}^{p}$ is generated by the images of $C'(K(p))$, $D_{1}(K(p))$ and $D_{2}(K(p))$. As in the proof of Proposition \ref{surj_center}, we consider the action of $\phi_{\boldsymbol{m}}$ on $E_{1}/E_{1}^{p}$. 

First, by our assumption, the image of $D_{1}(K(p))$ under $\phi_{\boldsymbol{m}}$ is trivial. If we choose an arbitrary ideal $\mathfrak{a}$ prime to $6p$ such that $\chi^{\boldsymbol{1}-\boldsymbol{m}}(\sigma_{\mathfrak{a}})-N\mathfrak{a} \not \equiv 0 \bmod p$, then it follows by a similar argument to that of Proposition \ref{surj_center} that the $\chi^{\boldsymbol{1}-\boldsymbol{m}}$-isotypic component of $E_{1}/E_{1}^{p}$ is generated by $\phi_{\boldsymbol{m}}(\theta_{\mathfrak{a}}(\omega))$ and $\phi_{\boldsymbol{m}}(\theta_{\mathfrak{a}}(\pi\omega))$. 

By Proposition \ref{elliptic_norm}, we have
\[
\phi_{\boldsymbol{m}} \left( N_{K(p)/K(\bar{\p})}(\theta_{\mathfrak{a}}(\omega)) \right)
= \phi_{\boldsymbol{m}}(\theta_{\mathfrak{a}}(\pi\omega)^{1-\sigma_{\p}^{-1}})=\left(1-\chi^{\boldsymbol{m}-\boldsymbol{1}}(\sigma_{\p}) \right)\phi_{\boldsymbol{m}}(\theta_{\mathfrak{a}}(\pi\omega)).
\] On the other hand, we also have
\[
\phi_{\boldsymbol{m}} \left( N_{K(p)/K(\bar{\p})}(\theta_{\mathfrak{a}}(\omega)) \right)
=[K(p):K(\p)]\phi_{\boldsymbol{m}}(\theta_{\mathfrak{a}}(\omega)).
\] since $\phi_{\boldsymbol{m}}\sigma=\phi_{\boldsymbol{m}}$ for every $\sigma \in \mathrm{Gal}(K(p)/K(\bar{\p}))$. Hence it follows that 
\[
  [K(p):K(\p)]\phi_{\boldsymbol{m}}(\theta_{\mathfrak{a}}(\omega))=\left(1-\chi^{\boldsymbol{m}-\boldsymbol{1}}(\sigma_{\p}) \right)\phi_{\boldsymbol{m}}(\theta_{\mathfrak{a}}(\pi\omega)).
\] Since the degree $[K(p):K(\p)]$ is prime to $p$, this equality implies that the $\chi_{2}^{1-m_{2}}$-isotypic component of $E_{1}/E_{1}^{p}$ is generated by $\phi_{\boldsymbol{m}}(\theta_{\mathfrak{a}}(\pi \omega))$ and that $\phi_{\boldsymbol{m}}(\theta_{\mathfrak{a}}(\omega))$ is nontrivial if and only if the term 
\[
  1-\chi_{2}^{m_{2}-1}(\sigma_{\p}) \tag{$\ast$}
\] is nonzero modulo $p$. It suffices to show that ($\ast$) is nontrivial under the assumption on the number of primes of $K(p)$ above $\p$, since then we can apply Lemma \ref{infty_to_one}. Note that the character $\chi_{2}^{w_{k}}$ induces an isomorphism
\[
  \mathrm{Gal}(K(\bar{\p})/K) \cong (O_{K}/\bar{\p})^{\times w_{K}}.
\] The assumption is equivalent to the condition that the Frobenius element $\sigma_{\p}$ generates $\mathrm{Gal}(K(\bar{\p})/K)$. Hence ($\ast$) is nonzero modulo $p$ as desired.
\end{proof}

Finally, we consider the case where $\boldsymbol{m} \geq (2,2)$ satisfies $\boldsymbol{m} \equiv (0,0) \bmod p-1$. Note that we cannot apply Lemma \ref{infty_to_one} since the first condition in the lemma is never satisfied in this case. We need the following lemma:

\begin{lemma}\label{coh}
  The following assertions hold.
  \begin{enumerate}
    \item The \'etale cohomology group 
    \[
      H^{1}_{\et}(O_{K(p^{\infty}), S_{p}}, \mathbb{Z}_{p}(p-1))^{\mathrm{Gal}(K(p^{\infty})/K)}
    \]  is isomorphic to $\mathbb{Z}_{p}$.
    \item The \'etale cohomology group 
    \[
      H^{1}_{\et}(O_{K, S_{p}},\mathbb{Z}_{p}(p-1))
    \] is isomorphic to $\mathbb{Z}_{p} \oplus \mathbb{F}_{p}$.
    \item Both the kernel and the cokernel of the restriction map
    \[
      H^{1}_{\et}(O_{K, S_{p}},\mathbb{Z}_{p}(p-1)) \to H^{1}_{\et}(O_{K(p^{\infty}), S_{p}}, \mathbb{Z}_{p}(p-1))^{\mathrm{Gal}(K(p^{\infty})/K)} 
    \] are isomorphic to $\mathbb{F}_{p}$.
  \end{enumerate}
\end{lemma}

\begin{proof}
  (1) We write the unique $\mathbb{Z}_{p}^{2}$-extension of $K$ by $K_{\infty}$. It holds by the Hochschild-Serre spectral sequence that
  \[
    H^{1}_{\et}(O_{K_{\infty}, S_{p}}, \mathbb{Z}_{p}(p-1))^{\mathrm{Gal}(K_{\infty}/K)} \xrightarrow{\sim} H^{1}_{\et}(O_{K(p^{\infty}), S_{p}}, \mathbb{Z}_{p}(p-1))^{\mathrm{Gal}(K(p^{\infty})/K)}
  \] (see Lemma \ref{k_vs_F} for a similar argument). This group can be identified with 
  \[
    \mathrm{Hom}_{\mathrm{Gal}(K_{\infty}/K)}(G_{K_{\infty}, S_{p}}^{(p)},  \mathbb{Z}_{p}(p-1))
  \] where $G_{K_{\infty}, S_{p}}^{(p)}$ is the Galois group of the maximal pro-$p$ extension of $K_{\infty}$ unramified outside $p$. Since $G_{K, S_{p}}^{(p)}$ is a free pro-$p$ group of rank two by Lemma \ref{free}, the group $G_{K_{\infty}, S_{p}}^{(p)}$ coincides with the commutator subgroup of $G_{K, S_{p}}^{(p)}$. Then it follows that the maximal abelian quotient of $G_{K_{\infty}, S_{p}}^{(p)}$ is a free $\mathbb{Z}_{p}[[\mathrm{Gal}(K_{\infty}/K)]]$-module of rank one by \cite[Theorem 2]{Ih1}. The assertion of (1) follows from this.

  (2-3) Since $G_{K_{\infty}, S_{p}}^{(p)}$ is a free pro-$p$ group, it follows that
  \[
    H^{2}_{\et}(O_{K, S_{p}},\mathbb{Z}_{p}(p-1)) \cong H^{2}(G_{K, S_{p}}^{(p)}, \mathbb{Z}_{p}(p-1))
  \] is trivial (cf. the beginning of Section \ref{4.2}). Hence by the five-term exact sequence, the kernel and the cokernel of the restriction map
  \[
    H^{1}_{\et}(O_{K, S_{p}},\mathbb{Z}_{p}(p-1)) \to H^{1}_{\et}(O_{K(p^{\infty}), S_{p}}, \mathbb{Z}_{p}(p-1))^{\mathrm{Gal}(K(p^{\infty})/K)} 
  \] are isomorphic to $H^{1}(\mathrm{Gal}(K(p^{\infty})/K), \mathbb{Z}_{p}(p-1))$ and $H^{2}(\mathrm{Gal}(K(p^{\infty})/K), \mathbb{Z}_{p}(p-1))$, respectively. Since both of these are isomorphic to $\mathbb{F}_{p}$, the assertions of (2) and (3) follow.
\end{proof}

\begin{proposition}\label{surj_center_3}
  The following assertions hold.
  \begin{enumerate}
    \item Assume that $\boldsymbol{m} \geq (2,2)$ satisfies $\boldsymbol{m} \equiv (0,0) \bmod p-1$ and the class number of $K(p)$ is not divisible by $p$. Then $\kappa_{\boldsymbol{m}}$ is surjective (in particular, nontrivial).
    \item Assume that $\kappa_{(p-1,p-1)}$ is surjective. Then the $\chi^{\boldsymbol{1}}$-component of the group of elliptic units mod $p$ 
    \[
      C(K(p))/\left( C(K(p)) \cap E_{1}^{p} \right)
   \] is two-dimensional.
  \end{enumerate}
\end{proposition}

\begin{proof}
  (1) We may assume $\boldsymbol{m}=(p-1,p-1)$. Recall that $\kappa_{\boldsymbol{m}}$ is an element of
  \[
    H^{1}_{\et}(O_{K(p^{\infty}), S_{p}}, \mathbb{Z}_{p}(p-1))^{\mathrm{Gal}(K(p^{\infty})/K)},
  \] which is isomorphic to $\mathbb{Z}_{p}$ by Lemma \ref{coh} (1). It suffices to show that $\kappa_{\boldsymbol{m}}$ generates this group. We fix an ideal $\mathfrak{a}$ prime to $6p$ such that \[
    \phi_{\boldsymbol{m}}\left( \theta_{\mathfrak{a}}(\omega) \right) \in E_{1}/E_{1}^{p}
  \] is nontrivial. Here, we use the assumption that the class number of $K(p)$ is not divisible by $p$. We claim that the image of the norm compatible system $\left( \theta_{\mathfrak{a}}(\omega_{1,n}+\omega_{2,n}) \right)_{n \geq 1}$ under the map $e_{\boldsymbol{m}}$, which we denote by $\kappa_{\boldsymbol{m}, \mathfrak{a}}$\footnote{Note that we used the same symbol in Section \ref{4.1} to denote the restriction of this cocycle.}, generates the free part of $H^{1}(O_{K, S_{p}},\mathbb{Z}_{p}(p-1)) \cong \mathbb{Z}_{p} \oplus \mathbb{F}_{p}$.

   The assertion of (1) follows from this claim. In fact, by this claim and Proposition \ref{kappa_a}, the element
  \[
    \left( N\mathfrak{a}-\chi^{\boldsymbol{1}-\boldsymbol{m}}(\sigma_{\mathfrak{a}}) \right) \kappa_{\boldsymbol{m}}
  \] is contained in the image of a generator of the free part of $H^{1}_{\et}(O_{K, S_{p}},\mathbb{Z}_{p}(p-1))$. Hence it follows that $N\mathfrak{a}-\chi^{\boldsymbol{1}-\boldsymbol{m}}(\sigma_{\mathfrak{a}}) \not \equiv 0 \bmod p^{2}$ and that $\kappa_{\boldsymbol{m}}$ is a generator since the cokernel of ($\ast$) is cyclic of order $p$ by Lemma \ref{coh} (3).
  
  First, $\kappa_{\boldsymbol{m}, \mathfrak{a}}$ is nonzero on the free part since its restriction is a nonzero multiple of $\kappa_{\boldsymbol{m}}$, which is nontrivial by Corollary \ref{Soule_cyc}. To prove that $\kappa_{\boldsymbol{m}, \mathfrak{a}}$ generates the free part, it suffices to show that its reduction to
  \[
    H^{1}_{\et}(O_{K, S_{p}}, \mathbb{F}_{p}) \xrightarrow{\sim}  H^{1}_{\et}(O_{K(p), S_{p}}, \mathbb{F}_{p})^{\mathrm{Gal}(K(p)/K)} \tag{$\ast$}
  \] is nontrivial. However, by the construction of $e_{\boldsymbol{m}}$, the image is the character associated to $p$-th roots of
  \[
    \phi_{\boldsymbol{m}}\left( \theta_{\mathfrak{a}}(\omega) \right) \in E_{1}/E_{1}^{p}
  \] which is nontrivial. This concludes the proof of (1).

  (2) By assumption, $\kappa_{(p-1,p-1)}$ generates $H^{1}_{\et}(O_{K(p^{\infty}), S_{p}}, \mathbb{Z}_{p}(p-1))^{\mathrm{Gal}(K(p^{\infty})/K)}$. Moreover, if we choose an ideal $\mathfrak{a}$ prime to $6p$ such that $N\mathfrak{a}-\chi^{\boldsymbol{1}-\boldsymbol{m}}(\sigma_{\mathfrak{a}}) \not \equiv 0 \bmod p^{2}$, then it follows from Proposition \ref{kappa_a} and Lemma \ref{coh} that $\kappa_{\boldsymbol{m}, \boldsymbol{a}}$ generates the free part of $H^{1}_{\et}(O_{K, S_{p}},\mathbb{Z}_{p}(p-1))$. 
  
  The long exact sequence associated to 
  \[
    0 \to \mathbb{Z}_{p}(p-1) \xrightarrow{p}  \mathbb{Z}_{p}(p-1) \to \mathbb{F}_{p} \to 0
  \] shows that the reduction map
  \[
    H^{1}_{\et}(O_{K, S_{p}}, \mathbb{Z}_{p}(p-1)) \to H^{1}_{\et}(O_{K, S_{p}}, \mathbb{F}_{p}) \cong \mathbb{F}_{p}^{2}
  \] is surjective. The latter cohomology group can be identified with the $\chi^{\boldsymbol{1}}$-isotypic component of $O_{K(p),S_{p}}^{\times}/O_{K(p),S_{p}}^{\times p}$ via ($\ast$) and the Kummer isomorphism.

  To conclude the proof, it suffices to show that the $p$-torsion subgroup of $H^{1}_{\et}(O_{K, S_{p}}, \mathbb{Z}_{p}(p-1))$ maps to the image of $\mu_{p}$. In fact, then the image of the reduction map is two-dimensional and is generated by $\phi_{\boldsymbol{m}}\left( \theta_{\mathfrak{a}}(\omega) \right)$ and $\mu_{p}$, both of which are contained in $C(K(p))$.
  
  Note that the $p$-torsion subgroup of $H^{1}_{\et}(O_{K, S_{p}}, \mathbb{Z}_{p}(p-1))$ coincides with $H^{0}_{\et}(O_{K, S_{p}}, \mathbb{F}_{p})=\mathbb{F}_{p}$ via the connecting homomorphism. The image of $1 \in \mathbb{F}_{p}$ in $H^{1}_{\et}(O_{K(p), S_{p}}, \mathbb{F}_{p})^{\mathrm{Gal}(K(p)/K)}$ is a homomorphism
  \[
    G_{K(p), S_{p}}^{(p)} \to \mathbb{F}_{p} \colon \sigma \mapsto \frac{\chi^{(p-1,p-1)}(\sigma)-1}{p} \bmod p
  \] which is trivial on $K(p)(\mu_{p^{2}})$ since $\chi^{(p-1,p-1)}$ is just the $(p-1)$-th power of the $p$-adic cyclotomic character. Hence the $p$-torsion subgroup of $H^{1}_{\et}(O_{K, S_{p}}, \mathbb{Z}_{p}(p-1))$ corresponds to $\mu_{p}$ via Kummer theory as desired.
\end{proof}

Now we prove Theorem \ref{thm:main} (3), which is an analogue of Theorem \ref{thm:fund} (3).

\begin{theorem}\label{thm:surj_22}
  The following two assertions are equivalent.
  \begin{enumerate}
    \item The elliptic Soul\'e character $\kappa_{\boldsymbol{m}}$ is surjective for every $\boldsymbol{m} \in I$ such that $\boldsymbol{m} \geq (2,2)$ unless $\boldsymbol{m} \equiv \boldsymbol{1} \bmod p-1$.
    \item The class number of $K(p)$ is not divisible by $p$ and there exists a unique prime of $K(p)$ above $\p$. 
  \end{enumerate}
\end{theorem}

\begin{proof}
  $(2) \Rightarrow (1)$ : Propositions \ref{surj_center}-\ref{surj_center_3}. $(1) \Rightarrow (2)$ : Let $\boldsymbol{m}=(m_{1}, m_{2}) \in I$ that satisfies $\boldsymbol{m} \geq (2,2)$ and $\boldsymbol{m} \not \equiv (0,0) \bmod p-1$. For an ideal $\mathfrak{a}$ prime to $6p$ such that $N\mathfrak{a}-\chi^{\boldsymbol{1}-\boldsymbol{m}}(\sigma_{\mathfrak{a}}) \not \equiv 0 \bmod p$, the element \[
    \epsilon_{\boldsymbol{m},1 \mathfrak{a}}
 \in K(p)^{\times}/K(p)^{\times p}
\] is nontrivial by Lemma \ref{infty_to_one}. As the proofs of Proposition \ref{surj_center} and Proposition \ref{surj_center_2} show, this nontriviality is equivalent to that of the $\chi^{\boldsymbol{1}-\boldsymbol{m}}$-isotypic component of $C(K(p))/\left( C(K(p)) \cap E_{1}^{p} \right)$. Moreover, the proof of Proposition \ref{surj_center_2} implies that $1-\chi_{2}^{m_{2}-1}(\sigma_{p})$ is nonzero modulo $p$ for every $m_{2} \geq 2$ such that $m_{2} \equiv 1 \bmod w_{K}$.Hence it follows that there exists a unique prime of $K(p)$ above $\p$.

Finally, since we assume that $\kappa_{(p-1,p-1)}$ is surjective, the $\chi^{\boldsymbol{1}}$-isotypic component of $C(K(p))/\left( C(K(p)) \cap E_{1}^{p} \right)$ is two-dimensional by Proposition \ref{surj_center_3} (2). Hence we have
\[
  \dim_{\mathbb{F}_{p}}C(p)/\left( C(K(p)) \cap E_{1}^{p} \right) \geq [K(p) : K]=\dim_{\mathbb{F}_{p}} E_{1}/E_{1}^{p}
\] by Dirichlet's unit theorem. The assertion now follows from Theorem \ref{Rubin_cl}.
\end{proof}

\subsection{Third case}\label{5.3}

We assume that either $m_{1}=1$ or $m_{2}=1$ and $\boldsymbol{m} \equiv (1,1) \bmod p-1$. It is enough to assume that $\boldsymbol{m}=(p,1)$ since a similar argument can be applied to the case of $(1,p)$.

We fix an ideal $\mathfrak{a}$ prime to $p$ such that $N\mathfrak{a}-1 \not \equiv 0 \bmod p$. The element
\[
\epsilon_{\boldsymbol{m}, 1, \mathfrak{a}} \in \left( K(p)^{\times}/K(p)^{\times p} \right)^{\mathrm{Gal}(K(p)/K)}=K^{\times}/K^{\times p}.
\] is nothing but the norm of $\theta_{\mathfrak{a}}(\bar{\pi}\omega)$ (times $w_{K}$), which is congruent to $\pi^{12w(N\mathfrak{a}-1)}$ modulo the image of $O_{K}^{\times}$ by \cite[Chapter $\rm{I\hspace{-.01em}I}$ \textsection 2, 2.2 Proposition and 2.5 Proposition $(\mathrm{ii})$]{dS}. Hence Lemma \ref{infty_to_one} gives:

\begin{proposition}\label{surj_(1,p)}
Let $m>1$ be an integer such that $m \equiv 1 \bmod p-1$. Then, the $(m,1)$-th (resp. the $(1,m)$-th) elliptic Soul\'e character $\kappa_{(m,1)}$ (resp. $\kappa_{(1,m)}$) is surjective (in particular, nontrivial).
\end{proposition}

\subsection{Fourth case}\label{5.4}

We assume that $\boldsymbol{m}=(m_{1}, m_{2}) \in I$ satisfies $m_{2}=1$ and $m_{1} \not \equiv 1 \bmod p-1$. We fix an ideal $\mathfrak{a}$ prime to $6p$ such that we have $N\mathfrak{a}-\chi_{1}^{1-m_{1}}(\sigma_{\mathfrak{a}}) \not \equiv 0 \bmod p$. 

\begin{proposition}\label{surj_edge}
Let $\boldsymbol{m}=(m_{1}, m_{2}) \in I$.
\begin{enumerate}
\item Suppose that $m_{2}=1$. If the class number of $K(\p)$ is not divisible by $p$, then $\kappa_{\boldsymbol{m}}$ is surjective (in particular, nontrivial).
\item Suppose that $m_{1}=1$. If the class number of $K(\bar{\p})$ is not divisible by $p$, then $\kappa_{\boldsymbol{m}}$ is surjective (in particular, nontrivial).
\end{enumerate}
\end{proposition}

\begin{proof} 
We only prove the first assertion since the proof of (2) is similar. By Lemma \ref{infty_to_one}, the assertion is equivalent to saying that $\epsilon_{\boldsymbol{m}, 1, \mathfrak{a}} \in K(\p)^{\times}/K(\p)^{\times p}$ is nontrivial. By Theorem \ref{Rubin_cl}, it follows that
\[
  C(K(\p))/\left( C(K(\p)) \cap E_{(1,0)}^{p} \right)= E_{(1,0)}/E_{(1,0)}^{p}
\] and $\chi_{1}^{1-m_{1}}$-isotypic component of this space is one-dimensional by Dirichlet's unit theorem.  

Note that the group $C(K(\p))$ is generated by $\theta_{\mathfrak{b}}(\bar{\pi}\omega)^{\sigma-1}$, where $\mathfrak{b}$ runs through the ideals of $O_{K}$ with $(\mathfrak{b}, 6\mathfrak{p})=1$ and $\sigma$ through $\mathrm{Gal}(K(\p)/K)$ (see  Remark \ref{useful_remark}). A similar argument to the proof of Proposition \ref{surj_center} shows that the $\chi_{1}^{1-m_{1}}$-isotypic component of $E_{(1,0)}/E_{(1,0)}^{p}$ is generated by $\epsilon_{\boldsymbol{m}, 1, \mathfrak{a}}$ when regarding it as a subspace of $K(\p)^{\times}/K(\p)^{\times p}$. This concludes the proof.
\end{proof}

Finally we prove Theorem \ref{thm:main} (5) in the introduction.

\begin{theorem}\label{thm:surj_edge}
The following two assertions are equivalent.
\begin{enumerate}
\item The elliptic Soul\'e characters $\kappa_{(m,1)}$ are surjective for all $(m,1) \in I$.
\item The class number of the mod-$\p$ ray class field $K(\p)$ is not divisible by $p$.
\end{enumerate}
A similar equivalence also holds for $K(\bar{\p})$.
\end{theorem}

\begin{proof}
 $(2) \Rightarrow (1)$ : Propositions \ref{surj_(1,p)}-\ref{surj_edge}.
 $(1) \Rightarrow (2)$ : The proof of Proposition \ref{surj_edge} shows that $\chi_{1}^{1-m}$-isotypic component of 
  \[
  C(K(\p))/\left( C(K(\p)) \cap E_{(1,0)}^{p} \right) \subset K(\p)^{\times}/K(\p)^{\times \p}
  \] is one-dimensional for every $(m,1) \in I$ such that $m \not \equiv 1 \bmod p-1$. Hence we have
  \[
    C(K(\p))/\left( C(K(\p)) \cap E_{(1,0)}^{p} \right)=E_{(1,0)}/E_{(1,0)}^{p}
  \] by Dirichlet's unit theorem. By Theorem \ref{Rubin_cl}, the class number of $K(\p)$ is prime to $p$ as desired. 
\end{proof}

{\bf Acknowledgements.} 
The author would like to thank Professor Akio Tamagawa and Benjamin Collas for their helpful comments and discussions. The author also would like to express his gratitude to the referees for valuable suggestions to improve the content of this paper. This work is supported by JSPS KAKENHI Grand Number 23KJ1882.

\bibliographystyle{amsplain}
\bibliography{References}

\end{document}